\documentclass[number,dvips]{arxbj}
\usepackage{upgreek}

\aid{0}
\volume{16}
\issue{3}
\pubyear{2010}
\firstpage{882}
\lastpage{908}
\doi{10.3150/09-BEJ238}

\makeatletter
\newcommand{\cut}[1]{}
\def\epsilon{\varepsilon}

\newtheorem{theorem}{Theorem}
\newtheorem{proposition}[theorem]{Proposition}
\newtheorem{lemma}[theorem]{Lemma}
\newtheorem{corollary}[theorem]{Corollary}

\newproclaim{definition}[theorem]{Definition}

\newremark{example}{Example}
\newremark{remarks}{Remarks}
\newremark{remark}{Remark}
\makeatother

\begin{document}
\begin{frontmatter}

\title{Quantitative bounds for Markov chain convergence: Wasserstein
and total~variation~distances}
\runtitle{Wasserstein and TV convergence of Markov chains}

\begin{aug}
\author[a]{\fnms{Neal} \snm{Madras}\corref{}\thanksref{a}\ead[label=e1]{madras@mathstat.yorku.ca}}
\and
\author[b]{\fnms{Deniz} \snm{Sezer}\thanksref{b}\ead[label=e2]{adsezer@math.ucalgary.ca}}
\runauthor{N. Madras and D. Sezer}
\address[a]{Department of Mathematics and Statistics, York University,
4700 Keele Street, Toronto, ON M3J 1P3, Canada. \printead{e1}}
\address[b]{Department of Mathematics and Statistics,
University of Calgary, 2500 University Drive, Calgary, AB T2N 1N4,
Canada. \printead{e2}}
\end{aug}

\received{\smonth{7} \syear{2008}}
\revised{\smonth{9} \syear{2009}}

%
\begin{abstract}
We present a framework for obtaining explicit bounds on the rate of
convergence to equilibrium of
a Markov chain on a general state space, with respect to both total
variation and Wasserstein distances. For Wasserstein bounds, our main
tool is Steinsaltz's convergence theorem for locally contractive
random dynamical systems.
We describe practical methods for finding
Steinsaltz's ``drift functions'' that prove local contractivity.
We then use the idea of ``one-shot coupling'' to derive criteria
that give bounds for
total variation distances in terms of Wasserstein distances.
Our methods are applied to two examples: a two-component
Gibbs sampler for the Normal distribution and a random logistic
dynamical system.
\end{abstract}

%

\end{frontmatter}

\section{Introduction}\label{sec1}

In many theoretical or applied problems involving positive recurrent
Markov chains, it is important to estimate the number of iterations
until the distribution of the chain is ``close'' to its
equilibrium distribution.
Suppose we have a Markov chain with state space $\chi$,
initial state $x$, transition
probability kernel $P$ and limiting
stationary distribution $\pi$.
We would like a quantitative bound such as
\[
d(P^{n}(x,\cdot),\pi(\cdot)) \leq g(x,n) ,
\]
where $d$ is a metric on the set of probability measures
and $g(x,n)$ is a function that can be computed explicitly.
For example, knowledge of such a function $g$ can be valuable to
Bayesian statisticians using Markov chain Monte Carlo (MCMC)
approximations because it
tells them how many MCMC steps will ensure a good
approximation to the posterior distribution under consideration.
An excellent survey on the theory of general state
space Markov chains and MCMC is~\cite{RobRos04}.

An important technical point is the specification of the metric $d$ on the
set of probability measures.
Two common choices are the \textit{total variation} (TV) \textit{metric}
(denoted $d_{\mathrm{TV}}$)
and the \textit{Wasserstein metric} (denoted $d_{\mathrm{W}}$);
see Section \ref{sec-metrics} for definitions and basic properties
of these two metrics. \cut{Informally, consider random variables $X_1$ and
$X_2$ on $\chi$ with respective probability measures $\mu_1$ and $\mu_2$.
Then the TV distance between $\mu_1$ and $\mu_2$ is the smallest
possible value of $\Pr\{X_1\neq X_2\}$ among all
couplings (joint distributions) of $X_1$ and $X_2$.
The Wasserstein distance between $\mu_1$
and $\mu_2$ is defined in terms of a metric $\rho$ on $\chi$; it is
the smallest possible expected distance $E(\rho(X_1,X_2))$ among all
couplings of $X_1$ and $X_2$. We shall sometimes
abuse our terminology slightly by referring to the [TV or Wasserstein]
distance between $X_1$ and $X_2$, by which
we mean the distance between $\mu_1$ and $\mu_2$.}

There is a rich literature on Markov chain convergence in total variation
distance. Many tools have been developed for convergence in
TV, involving probabilistic methods (for example,\ coupling, strong
uniform times;
see \cite{Diac,Jerrum,RobRos04} for reviews),
analytic methods (spectral analysis, Fourier analysis,
operator theory; see \cite{Diac,Saloff})
and geometric methods (path bounds, isoperimetry; see \cite{Jerrum,Saloff}).
Much of the progress, and many of the sharpest results,
have been for discrete state spaces \cite{Diac,Jerrum,Saloff},
including spaces related to graphs, algebraic structures, or models from
statistical physics. Some results extend to general state spaces,
but some basic discrete properties and methods do not have
convenient analogs in the general case. Continuous state spaces
are of particular interest in Bayesian MCMC applications
\cite{Gilks,RobRos04}, but quantitative rigorous results about realistic
examples are scarce.

Frequently, the desirable functions $g$ to seek are of the
form $g(x,n)=C(x)r^{n}$, where $C(x)$ and $r$ can be computed
explicitly. The existence of such a function for the TV metric is
called \textit{geometric ergodicity} and is known to hold under fairly
general conditions (see, for example,\ \cite{Meyn,RobRos97}). Explicit
identification of such functions can be an intricate task, however.
A classical result in this context is due to Doeblin: if
there exists a probability measure $\nu$ and $0<\epsilon<1$ such
that $P(x,\mathrm{d}y)\geq\epsilon\nu(\mathrm{d}y)$ for every $x$, then
$d_{\mathrm{TV}}(P^{n}(x,\cdot),\pi)\leq(1-\epsilon)^n$. It is possible
to get similar bounds using coupling when Doeblin's condition holds
only on a subset $K$, if a ``drift function'' to $K$ exists. More
precisely, one needs
(i) $P(x, \mathrm{d}y)\geq\epsilon\nu(\mathrm{d}y)$ for all $x\in K$;
(ii) a function $V>1$ and a constant $\alpha>1$ such that
$E(V(Y_{n+1})|Y_{n}=y)
<V(y)/\alpha$ for all $y\in K^c$. These conditions are called
(i) \textit{minorization} and (ii) \textit{drift conditions}
\cite{Meyn,Ros95}. For practitioners who want to
implement these conditions, the challenge is to identify such a set
$K$ and a drift function $V$ that lead to tractable calculations and
good results. See
\cite{HobGey98} for an impressive application of these conditions to a
Bayesian random effects model. A good survey and another realistic
application is in \cite{JoHo}.

Coupling arguments for proving TV bounds typically use two coupled
versions of a Markov chain that coalesce relatively quickly.
This is often technically easier to do in discrete state spaces
than in state spaces with no atoms.
Minorization and drift conditions offer one solution to
this difficulty: coalescence
is facilitated when the coupled chains are simultaneously in the set $K$.
However, in many situations, it may be hard to force
coupled chains to coalesce, but it may be easier
to force them to come (and stay) very close to each other.
Closeness of two chains in the metric of the state
space roughly corresponds to closeness of their
distributions in the Wasserstein distance.
For this reason,
the Wasserstein distance can be a tractable alternative to the total
variation distance for problems in continuous state spaces
(see, for example, \cite{Gibbs}).
Although Wasserstein convergence can be weaker than TV convergence,
we shall show that
under certain conditions, bounds on the rate of Wasserstein convergence
can be used to get bounds on the rate of TV convergence
(see Section 4). Thus, proving Wasserstein convergence is sometimes a
step toward proving TV convergence. Huber \cite{Huber} also uses
this general philosophy, employing rather different methods from ours.

A particularly successful framework for studying convergence in
Wasserstein distance is random dynamical systems, or
iterated function systems \cite{DiacFreed99,Steinsaltz}. An iterated
function system is a sequence of random maps of the form
$F_{n}(x)=f_{1}\circ f_2\circ\cdots\circ f_n(x)$ or $\tilde
{F}_{n}(x)=f_{n}\circ f_{n-1}\circ\cdots\circ f_1(x)$,
where $f_{1},f_{2},\ldots$ are independent and identically
distributed (i.i.d.) random maps. (Two examples are described later
in this section.) The sequence
$\{\tilde{F}_{n}(x) \dvtx n\geq1\}$ is called the
\textit{forward sequence} and is a Markov chain. Many examples of Markov
chains can be represented as forward iterates of i.i.d. random maps.
$\{F_{n}(x) \dvtx n\geq1\}$ is called the \textit{backward} sequence and,
under certain
conditions, it converges pointwise to a random variable, $X_\infty$,
independent of the starting point $x$. If $X_\infty$ exists, in
which case the system is called \textit{attractive}, the distribution of
$X_{\infty}$ is also the stationary distribution $\pi$ of the Markov chain
$\tilde{F}_{n}(x)$. The rate at which $E[\rho(F_{n}(x),X_\infty)]$
converges to zero is an upper bound on the rate of convergence in
distribution of the Markov chain $\tilde{F}_{n}(x)$ to $\pi$ in Wasserstein
distance. Indeed, since $F_{n}(x)$ has distribution $P^n(x,\cdot)$
(as does $\tilde{F}_{n}(x)$) and since $X_{\infty}\sim\pi$,
we have
%
\begin{equation}
d_{\mathrm{W}}(P^n(x,\cdot),\pi) \leq E[\rho(F_{n}(x),X_\infty)] .
\end{equation}

One condition that guarantees attractivity is strong contractivity,
that is, $E[\log\operatorname{Lip} f]<0$, where $\operatorname{Lip} f$ is the Lipschitz
constant of the (random) function $f$.
This condition is a generalization of the stronger condition that
there exists a constant $r\in(0,1)$ such that $\rho(f(x),f(y)) \leq
r\rho(x,y)$ for all $x$ and $y$, with probability 1.
(Gibbs \cite{Gibbs} used a variation of this condition to get a bound
for the Wasserstein distance of a Markov chain $X _{n}$ to its
stationary distribution using coupling. See also \cite{DiacFreed99}
for a related result.)
However, applications frequently require weaker conditions.
Steinsaltz \cite{Steinsaltz} proves
attractivity under a more general condition, called ``local
contractivity'', which says that there exists a ``drift function''
$\phi\dvtx\mathcal{X}\mapsto[1,\infty)$ and a constant $r\in(0,1)$ such that
\[
G_{n}(x) := E[D_x F_{n}] \leq \phi(x)r^n ,
\]
where $D_xf:=\limsup_{y\rightarrow
x}\frac{\rho(f(x),f(y))}{\rho(x,y)}$. He proves that if local
contractivity holds, then
\[
E[\rho(F_{n}(x),X_\infty)] \leq C_{x}r^{n}\qquad
\mbox{for every }n\geq1,
\]
where $C_{x}$ is a number that can be computed explicitly; see Section
\ref{subsec-contrac} for further discussion.
Steinsaltz's use of the term ``drift'' is analogous to, but different
from, Rosenthal's use (which, in turn, is closely related to Foster--Lyapunov
functions; see \cite{FMM} for a review and references).

Like the minorization and drift conditions, the local contractivity
condition requires preliminary work to obtain a drift function.
The goal of the first part of this paper (Section \ref{sec3}) is to provide
a systematic framework for doing this.

We developed our methods using two examples. The first is a
simple Gibbs sampler chain for Bayesian estimation of the mean and
variance of a Normal distribution. The second example is a
randomized version of the classical logistic map from dynamical
systems theory.

The paper is organized as follows. The remainder of this section is
devoted to descriptions of our two main examples.
Section \ref{sec-metrics} provides definitions and basic properties of the
Wasserstein and total variation metrics.
Section \ref{sec3} examines the task of finding a drift function that
produces quantitative bounds on Wasserstein convergence.
Section \ref{subsec-contrac} reviews the results
of Steinsaltz \cite{Steinsaltz} and Section \ref{sec3.2} presents an approach to
finding drift functions by looking for sub-eigenfunctions of
a certain dominating operator. Section \ref{sec3.3} then uses this approach to find
drift functions for our Gibbs sampler example.
Section \ref{sec-WTV} shows how bounds on the Wasserstein metric may be ``upgraded''
to bounds on the total variation metric in some situations.
Section \ref{sec4.1} reviews the idea of ``one-shot coupling'' \cite{RobRos02} and
presents our key technical result (Theorem \ref{prop-WTVgen}).
Sections~\ref{sec4.2} and \ref{sec4.3} apply this result to our two examples.

\begin{example}[(Normal Gibbs sampler)]
A simple Bayesian estimation problem is the following.
Consider a random sample of size $J$ from the Normal distribution
with mean $\theta$ and variance~$\sigma^2$ (written $N(\theta,\sigma^2)$).
We assume that
$\theta$ and $S:=\sigma^{-2}$ are themselves independent random
variables from
Normal and Gamma prior distributions respectively:
\[
\theta\sim N(\xi,K^{-1})
\quad\mbox{and}\quad
S:=\sigma^{-2} \sim\Gamma(\alpha,\beta).
\]
(Here, $\Gamma(\alpha,\beta)$ is the Gamma distribution with density
$s^{\alpha-1}\beta^{\alpha}\exp(-\beta s)/\Gamma(\alpha)$.)
Let $Y:=Y_1,\ldots,Y_J$ be our random sample from $N(\theta,\sigma^2)$
(conditionally independent, given $\theta$ and $\sigma$).
The joint posterior for $\theta$ and $S$ given $Y$ is
%
\begin{equation}
\label{GS-post}
p(\theta,s|Y) \propto s^{\alpha-1+J/2}
\exp\biggl[-\beta s -\frac{K(\theta-\xi)^2}{2}-\frac{s\sum(Y_j-\theta)^2}{2}
\biggr]
\end{equation}
(where $\sum$ is the sum over $j$ from 1 to $J$).
Besides positive values of $K$, we shall also consider the case $K=0$.
When $K=0$, the prior for $\theta$ is not a probability distribution;
however, the joint posterior \textit{is} a probability distribution. (We
can view
$K=0$ as the ``flat prior'' limit $K\rightarrow0+$.
The case $\beta=0$ is similar.)
The \textit{Gibbs sampler} is the Markov chain $(\theta_t, S_t)$ defined
recursively by drawing $\theta_{t}$ from its conditional distribution
given $Y$ and $S=S_{t-1}$, followed by drawing $S_{t}$ from its conditional
distribution given $Y$ and $\theta=\theta_{t}$:
\begin{eqnarray*}
\theta_{t} & \sim& N \biggl( \frac{S_{t-1}\sum Y_j+K\xi}{S_{t-1}J+K} ,
\frac{1}{S_{t-1}J+K} \biggr) , \\
S_{t} & \sim& \Gamma\biggl( \alpha+\frac{J}{2} ,
\beta+\frac{1}{2}\sum(Y_j-\theta_{t})^2 \biggr) .
\end{eqnarray*}
We can represent this procedure as follows:
%
\begin{eqnarray}
\label{GS.A1th}
\theta_{t} & = & \frac{Z_t}{\sqrt{S_{t-1}J+K}} +
\frac{S_{t-1}\sum Y_j+K\xi}{S_{t-1}J+K} ,
\qquad\mbox{where $Z_t\sim N(0,1)$}, \\
\label{GS.A1}
S_{t} & = & \frac{G_t}{\beta+\frac{1}{2}\sum(Y_j-\theta_{t})^2} ,
\qquad\mbox{where $G_t\sim\Gamma(\alpha+J/2,1)$ }
\end{eqnarray}
(and $\{Z_t\}$ and $\{G_t\}$ are independent i.i.d.\ sequences).
Let
\[
\bar{Y} = \frac{1}{J} \sum_{j=1}^J Y_j
\quad\mbox{and}\quad
\Sigma_0 = \beta+\frac{1}{2}\sum_{j=1}^J(Y_j-\bar{Y})^2
\]
(we treat these as constants, since we always condition on $Y$).
Since
%
\begin{equation}
\label{GS.SS}
\sum_{j=1}^J(Y_j-\theta)^2 = \sum_{j=1}^J(Y_j-\bar{Y})^2 +
J(\bar{Y}-\theta)^2 ,
\end{equation}
we can write equation (\ref{GS.A1}) as
%
\begin{equation}
\label{GS.A2}
S_{t} = \frac{G_t}{\Sigma_0+
({J}/{2})(\bar{Y}-\theta_{t})^2} .
\end{equation}
Using equation (\ref{GS.A1th}), we can express
(\ref{GS.A2}) as a random dynamical system, as follows:
%
\begin{equation}
\label{GS.Ards}
S_{t} = f_{t}(S_{t-1}) ,\qquad t=1,2,\ldots,
\end{equation}
where $f_{t}\dvtx(0,\infty)\rightarrow(0,\infty)$ is the random function
%
\begin{equation}
\label{GS.Afdef}
f_{t}(s) = \frac{G_t}{\Sigma_0+
({J}/{2}) ( {Z_t}/{\sqrt{sJ+K}} +
{(\xi-\bar{Y})K}/({sJ+K}) )^2 }
\end{equation}
with the random variables $G_t$ and $Z_t$ as above. The case $K=0$
is of special interest (representing an improper prior for $\theta$)
and equation (\ref{GS.Afdef}) specializes to
%
\begin{equation}
\label{GS.Afdef0}
f_{t}(s) = \frac{G_t}{\Sigma_0+Z_t^2/(2s) } .
\end{equation}
We note that the posterior (\ref{GS-post}) is a proper probability
distribution when $K=0$, even though the prior is not (to see this, use
(\ref{GS.SS}) and integrate $\theta$ first).

Without loss of generality,
we can assume that $\xi$ is zero and that $K$ is either 0 or 1.
(Indeed, if $K>0$,
then we can let $\tilde{\theta}=(\theta-\xi)\sqrt{K}$,
$\tilde{Y_i}=(Y_i-\xi)\sqrt{K}$, $\tilde{\sigma}^2=K\sigma^2$ and
$\tilde{\beta}=K\beta$; then
$\tilde{Y_i}\sim N(\tilde{\theta},\tilde{\sigma}^2)$, where
$\tilde{\theta}\sim N(0,1)$ and
$\tilde{\sigma}^{-2}\sim\Gamma(\alpha,\tilde{\beta})$.)
Accordingly, for our Markov chain $\{S_t\}$ with $K\in\{0,1\}$,
let $P_K$ be the chain's transition probability kernel,
let $p_K(\cdot,\cdot)$ be
the density of $P_K$ and let $\pi_K$ be the stationary distribution.

We shall obtain quantitative bounds for the convergence of
our Gibbs sampler chain $P_K$ ($K\in\{0,1\}$); see Propositions
\ref{prop-K1bd} and \ref{prop-GSTVW}, and the
discussions of numerical results following each.
Roberts and Rosenthal \cite{RobRos02} analyzed this chain with
flat priors, that is,\ $K=\xi=\beta=0$ and $\alpha=1$. In particular,
their results show that
$\limsup_{n\rightarrow\infty} [d_{\mathrm{TV}}(P_0^n(x,\cdot),\pi_0) ]^{1/n}
\leq1/J$. This would equal our asymptotic rate if we could replace
$w$ by 1
in Proposition \ref{prop-GSTVW}. The analysis of \cite{RobRos02}
uses the property
that the recursion for $1/S_t$ is a linear function of $1/S_{t-1}$,
which only holds when $K=0$.
Their approach cannot handle the case $K>0$. Our method
of Section 4 may be viewed as a more powerful (nonlinear) generalization
of \cite{RobRos02}.
\end{example}

\begin{example}[(Random logistic map)]
We consider the i.i.d.\ random maps $f_1,f_2,\ldots$ on $[0,1]$
defined by
\[
f_i(x) = 4 B_i x(1-x),
\]
where $B_1,B_2,\ldots$ are i.i.d. random variables having the
$\operatorname{Beta}(a+\frac{1}{2},a-\frac{1}{2})$ distribution.
Here, $a>\frac{1}{2}$ is a fixed number.
It is known
that the $\operatorname{Beta}(a,a)$ distribution is the unique stationary distribution for
this iterated function system \cite{ChaLet}. Our result for this
example will
provide bounds that are more qualitative than quantitative.
Asymptotic convergence properties of this example have been studied in the
literature. Steinsaltz \cite{Steinsaltz}
showed that the system is locally contractive if $a\geq2$
and hence that the corresponding Markov chain converges to equilibrium
exponentially rapidly in the Wasserstein distance.
Using the techniques of Section 4, we shall prove the following theorem.
\end{example}

\begin{theorem}
\label{thm-logconv}
Assume that $a>1/2$ and let $x\in(0,1)$.
There then exists a constant $\tilde{C}_{a}$, depending only on $a$,
such that
\[
d_{\mathrm{TV}}(\tilde{F}_n(x),\beta_{a,a}) \leq \tilde{C}_{a}
[ d_{\mathrm{W}}(\tilde{F}_{n-1}(x),\beta_{a,a}) ]^{a/(a+1)}\qquad
\mbox{for all $n\geq1$}
\]
(where $\beta_{a,a}$ is a random variable having the
$\operatorname{Beta}(a,a)$ distribution).
\end{theorem}

Note that Theorem \ref{thm-logconv} does not assume local
contractivity (indeed,
local contractivity fails if $1/2<a<1$, by Corollary 3 of \cite{Stein2}
and Theorem 1 of \cite{Steinsaltz}).

Theorem \ref{thm-logconv} implies the following.
Assume that the random logistic Markov chain
$\{\tilde{F}_n(x) \dvtx n=0,1,\ldots\}$
converges to its equilibrium exponentially rapidly in Wasserstein distance,
that is, that\ there exists a constant $\rho\in(0,1)$ such that
%
\begin{equation}
\label{eq-LWexp}
\limsup_{n\rightarrow\infty}
[ d_{\mathrm{W}}(\tilde{F}_n(x),\beta_{a,a}) ]^{1/n} \leq \rho.
\end{equation}
It then also converges exponentially rapidly in TV distance, perhaps
at a modestly slower rate:
\[
\limsup_{n\rightarrow\infty} [
d_{\mathrm{TV}}(\tilde{F}_n(x),\beta_{a,a}) ]^{1/n} \leq
\rho^{a/(a+1)} < 1.
\]
Since the state space $(0,1)$ has diameter 1, we trivially have
$d_{\mathrm{W}}(\tilde{F}_n(x),\beta_{a,a}) \leq
d_{\mathrm{TV}}(\tilde{F}_n(x),\beta_{a,a})$. Hence, we conclude that
for $a>1/2$, our random logistic Markov chain converges to the
equilibrium exponentially rapidly in Wasserstein distance
\textit{if and only if} it converges exponentially rapidly in
TV distance.

\section{Wasserstein and total variation metrics}
\label{sec-metrics}
In this section, we review the definitions and some properties of two
metrics on the space of probability measures: the Wasserstein metric
and the total variation (TV) metric. For a broader review of metrics
on probabilities, see \cite{GibbsSu}.

Let $(\chi,\rho)$ be a complete separable metric space.
Consider two probability measures, $\mu_1$ and $\mu_2$,
on $\chi$. Let $\operatorname{Joint}(\mu_1,\mu_2)$ denote the set of all
probability measures $M$ on $\chi\times\chi$ whose marginal distributions
are $\mu_1$ and $\mu_2$, that is,
\[
\mu_1(\mathrm{d}x) = \int_yM(\mathrm{d}x,\mathrm{d}y)
\quad\mbox{and}\quad
\mu_2(\mathrm{d}y) = \int_xM(\mathrm{d}x,\mathrm{d}y) .
\]
In other words, if two random variables $X_1$ and $X_2$ have distributions
$\mu_1$ and $\mu_2$, respectively, then
$\operatorname{Joint}(\mu_1,\mu_2)$ is the set of all ``couplings'' of $X_1$ and $X_2$.

The \textit{Wasserstein distance} between $\mu_1$ and $\mu_2$, denoted
$d_{\mathrm{W}}(\mu_1,\mu_2)$, is defined to be
%
\begin{equation}
d_{\mathrm{W}}(\mu_1,\mu_2) =
\inf\biggl\{ \int_{\chi}\int_{\chi}\rho(x,y) M(\mathrm{d}x,\mathrm{d}y)
\dvtx M\in\operatorname{Joint}(\mu_1,\mu_2) \biggr\}.
\end{equation}
In other words, $d_{\mathrm{W}}(\mu_1,\mu_2)$ is the infimum of $E(\rho(X_1,X_2))$
over all couplings of $X_1$ and $X_2$ (where $X_i\sim\mu_i$).
It can be shown that there exists an $M$ that attains the infimum
(see, for example, Section 5.1 of \cite{Chen}).

The \textit{total variation (TV) distance} between $\mu_1$ and $\mu_2$, denoted
$d_{\mathrm{TV}}(\mu_1,\mu_2)$, is defined to be
%
\begin{equation}\label{eq-TVdef}
d_{\mathrm{TV}}(\mu_1,\mu_2) =
\sup\{ |\mu_1(A)-\mu_2(A)| \dvtx A\subset\chi\}.
\end{equation}
This sup is attained by some set $A$ (by
the classical Hahn decomposition for the signed measure $\mu_1-\mu_2$).
An equivalent definition of $d_{\mathrm{TV}}$ is
%
\begin{equation}\label{eq-TVcouprep}
d_{\mathrm{TV}}(\mu_1,\mu_2) =
\inf\bigl\{ M\bigl(\{(x,y)\dvtx x\neq y\}\bigr) \dvtx M\in\operatorname{Joint}(\mu_1,\mu_2)\bigr\}.
\end{equation}
In other words, $d_{\mathrm{TV}}(\mu_1,\mu_2)$ is the infimum of $\Pr\{X_1\neq
X_2\}$
over all couplings of $X_1$ and $X_2$ (where $X_i\sim\mu_i$).
For convenience, we shall sometimes talk about the Wasserstein or TV distance
between two random variables, which means the same thing as the
Wasserstein or TV distance between their distributions.

The following is relatively well known (see, for example, Theorem 5.7
of \cite{Chen}
or Proposition 3 of \cite{RobRos04}).

\begin{proposition}\label{prop-Wbasic}
Assume that $\mu_1$ and $\mu_2$ are probability measures on $\chi$,
having density functions~$p_1$ and $\rho_2$, respectively, with respect
to a common reference measure $\lambda$. Then
%
\begin{eqnarray}
\label{eq-TVpdf}
d_{\mathrm{TV}}(\mu_1,\mu_2) & = & \frac{1}{2}\int_{\chi}|p_1(z)-p_2(z)|
\lambda(\mathrm{d}z) \\
\label{eq-TVpdf1}
& = & \int_{z\dvtx p_1(z)>p_2(z)} \bigl(p_1(z)-p_2(z)\bigr) \lambda(\mathrm{d}z) \\
\label{eq-TVpdf2}
& = & 1 - \int_{\chi}\min\{p_1(z),p_2(z)\} \lambda(\mathrm{d}z) .
\end{eqnarray}
\end{proposition}

If the state space $\chi$ is bounded, then
$d_{\mathrm{W}}(\mu_1,\mu_2) \leq d_{\mathrm{TV}}(\mu_1,\mu_2) \times
[\sup\{\rho(x,y)\dvtx x,y\in\chi\}]$
and, in particular, TV convergence implies Wasserstein convergence.
However, in general, neither convergence implies the other. For example,
in $\mathbb{R}$, let $\mu_n$ be the two-point probability distribution
that has
$\mu_n(\{0\})=1-n^{-1}$ and $\mu_n(\{n\})=n^{-1}$. Then $\mu_n$
converges to
the point mass at 0 in the TV metric, but not in Wasserstein.
Also, let $\nu_n$ be the probability distribution on $[0,1]$ with density
$1+\sin(2\uppi nx)$; then $\nu_n$ converges to the uniform distribution
on $[0,1]$ in Wasserstein, but not in TV.

The following result will be very useful in Section \ref{sec-WTV}.

\begin{lemma}\label{lem-gABC}
Consider a deterministic measurable function $g \dvtx A\times B\rightarrow C$.
Let $W_1$ and $W_2$ be two $B$-valued random variables and let
$U$ be an $A$-valued random variable that is independent of both $W_i$'s.
Define the $C$-valued random variables $X_1$ and $X_2$ by
$X_i=g(U,W_i)$, $i=1,2$. Then
\[
d_{\mathrm{TV}}(X_1,X_2) \leq d_{\mathrm{TV}}(W_1,W_2) .
\]
\end{lemma}

\begin{pf}
Choose a joint distribution $M(\mathrm{d}w_1,\mathrm{d}w_2)$ of
a random vector $(\tilde{W}_1,\tilde{W_2})$ on $B\times B$ such that
$\tilde{W}_i \stackrel{d}{=}W_i$ for $i=1,2$ and
$M\{ \tilde{W}_1\neq\tilde{W}_2 \} = d_{\mathrm{TV}}(W_1,W_2)$.
Also, make $(\tilde{W}_1,\tilde{W_2})$ independent of $U$ and let
$\tilde{X}_i = g(U,\tilde{W}_i)$. Then $\tilde{X}_i \stackrel{d}{=}X_i$
for $i=1,2$, so
\[
d_{\mathrm{TV}}(X_1,X_2) \leq M\{ \tilde{X}_1\neq\tilde{X}_2 \}
\leq M\{ \tilde{W}_1\neq\tilde{W}_2 \} = d_{\mathrm{TV}}(W_1,W_2)
.
\]
\upqed
\end{pf}

\section{Convergence in the Wasserstein metric}\label{sec3}
\subsection{Local contractivity condition and
a convergence theorem} \label{subsec-contrac}
Our main tool to obtain quantitative bounds
for convergence in Wasserstein metric will be Steinsaltz's local
contractivity convergence theorem \cite{Steinsaltz}. Below, we
review this result in a form convenient for us.

\begin{definition}
An iterated function system is locally contractive if there exists
a function $\phi\dvtx\mathcal{X}\mapsto[1,\infty)$ and
$r\in(0,1)$ such that
\[
G_{n}(x) := E[D_x F_{n}] \leq \phi(x)r^n\qquad
\mbox{for all $n\geq1$},
\]
where $D_xf:=\limsup_{y\rightarrow x}\frac{\rho(f(x),f(y))}{\rho(x,y)}$.
If this holds, then $\phi$ is called a \textup{drift function}.
\end{definition}
\begin{theorem}\label{theo:steinsaltz}
If an iterated function system is locally contractive with a drift
function $\phi$ and if
\[
C_{x}:=E \Bigl[\rho(f(x),x)\sup_{0\leq t\leq1}\bigl\{\phi\bigl(x+t\bigl(f(x)-x\bigr)\bigr)\bigr\} \Bigr]
< \infty,
\]
then the system is attractive (in particular, $F_{\infty}(x)$ is independent
of $x$) and
\[
d_{\mathrm{W}}(F_{n}(x),F_\infty(x)) \leq
E\rho(F_{n}(x),F_\infty(x)) \leq
\frac{C_{x}r^n}{1-r} \qquad\mbox{for every $x\in\chi$}.
\]
%
\end{theorem}

Steinsaltz \cite{Steinsaltz} also gives a sufficient condition,
called the \textit{growth condition}, for a function $\phi$ to be a drift
function:
a continuous function $\phi\dvtx\mathcal{X}\mapsto[1,\infty)$ is a
drift function if $r<1$, where
\[
r := \sup_{x}E \biggl[\frac{\phi(f(x))}{\phi(x)}D_xf \biggr].
\]
Here is a short argument (different from the original proof in \cite
{Steinsaltz})
to explain why. Let $\mathcal{L}$ be the positive linear operator which
maps a generic function $g$ to the function
$\mathcal{L}(g)(x)=E[g(f(x))D_xf]$. Then
$G_{n}(x)=\mathcal{L}^{n}(1)(x)$, with $1$ here being the constant
function equal to $1$. Note that the
growth condition is equivalent to $\mathcal{L}\phi\leq r\phi$.
We will refer to any $\phi>0$ satisfying $\mathcal{L}\phi\leq r\phi$ as an
$r$\textit{-sub-eigenfunction} for $\mathcal{L}$. Now, if $\phi\geq1$ and
$\phi$ is an $r$-sub-eigenfunction, then $G_n(x)=\mathcal{L}^{n}1\leq
\mathcal{L}^{n}\phi\leq r^n\phi$ and hence $\phi$ is a drift function
with rate~$r$.

We note that Proposition 8 of \cite{Stein2} shows that the existence of a
$\phi$ satisfying the growth condition is also necessary for local
contractivity.

\subsection{How to apply the local contractivity convergence theorem:
Finding a drift function}\label{sec3.2}

Applying Steinsaltz's local contractivity convergence theorem to a
specific problem would be easy if one knew how to write down a
drift function. Here, we will propose two practical strategies that
can help us to do this.

The first strategy is to find a linear operator
$\tilde{\mathcal{L}}$ that dominates $\mathcal{L}$ and is simpler
to manage. If $\phi$ is an $r$-sub-eigenfunction for $\tilde{\mathcal{L}}$,
then it is an $r$-sub-eigenfunction for $\mathcal{L}$ as well.

One kind of operator that we can manage is defined as follows: let
$\{A_{i}\}_{i=1}^{n}$ be a finite partition of the state space $\chi$
and let
%
\begin{equation}\label{eq:form}
\tilde{\mathcal{L}}\phi(x)=b(x)\sum_{i=1}^{n}1_{A_i}(x)
\int_{\chi}\phi(s)\mu_i(\mathrm{d}s) ,
\end{equation}
where $b(x)$ is a positive function and each $\mu_i$ is a non-zero
finite measure on $\mathcal{\chi}$.

\begin{theorem} \label{theo:subeigen}
Let $\tilde{\mathcal{L}}$ be an operator of the form
(\ref{eq:form}). In order for $\tilde{\mathcal{L}}$ to have an
r-sub-eigenfunction, it is necessary and sufficient that the matrix
\[
Q(i,j)=\int_{A_{j}}b(x)\mu_i(\mathrm{d}x)
\]
has an $r$-sub-eigenvector $p=(p_1,p_2,\ldots,p_n)^t$, that is, $p_i>0$
$\forall i$ and $Qp\leq rp$. Moreover, if $p$ is an
$r$-sub-eigenvector for $Q$, then the function
%
\begin{equation}
\phi(x)=\sum_{j=1}^{n}p_j 1_{A_j}(x)b(x) \label{eq:subeigen}
\end{equation}
(and any positive multiple of it) is an $r$-sub-eigenfunction for
$\tilde{\mathcal{L}}$.
\end{theorem}
\begin{pf}
If $\phi$ is an $r$-sub-eigenfunction of
$\tilde{\mathcal{L}}$, then $b(x)\sum_{j=1}^{n}1_{A_j}(x)\int\phi(\mathrm{d}c)\mu
_{j}(c)\leq r\phi(x)$,
by definition of $\tilde{\mathcal{L}}$. Integrating both sides with
respect to $\mu_{i}$ gives
\[
\sum_{j=1}^{n} \int_{A_{j}} b(x)\mu_i(\mathrm{d}x)\int\phi(c)\mu_j(\mathrm{d}c)\leq
r\int\phi(x)\mu_i(\mathrm{d}x) .
\]
Therefore, the vector $p$ defined by $p_{i}:=\int\phi\mu_i$
is an $r$-sub-eigenvector for $Q$. Conversely, if $p$ is an
$r$-sub-eigenvector for $Q$ and if $\phi$ is as defined in
(\ref{eq:subeigen}), then
\[
\tilde{\mathcal{L}}\phi(x) = b(x)\sum_{i=1}^{n}1_{A_i}(x)
\sum_{j=1}^{n}p_{j}\int_{A_{j}}b(s)
\mu_i(\mathrm{d}s) \leq b(x)\sum_{i=1}^{n}1_{A_i}(x)rp_{i} = r\phi(x).
\]
Hence $\phi$ is an $r$-sub-eigenfunction and so is any
positive multiple of it.
\end{pf}

For the case $n=1$, Theorem \ref{theo:subeigen} implies the following.

\begin{corollary}
\label{cor:subeigen}
Assume that $b$ is a positive function and $\mu$ is a finite measure such
that $\mathcal{L}\phi(x) \leq b(x)\int_{\chi} \phi(s) \mu(\mathrm{d}s)$ for
every $x\in\chi$ and every positive $\phi$. Let $r=\int b(s)\mu(\mathrm{d}s)$.
Then $b$ is an $r$-sub-eigenfunction for $\mathcal{L}$.
\end{corollary}

Note that for an $r$-sub-eigenfunction $\phi$ to be a drift
function, it must be greater than $1$. If $\phi$ is bounded
away from $0$, we can get a drift function simply by scaling $\phi$.
However, if $\phi$ is not bounded away from $0$, we first need to
truncate it, as in the following lemma.

\begin{lemma} \label{lem:trunc} Let $\phi$ be an
$r$-sub-eigenfunction for $\mathcal{L}$. Let $\epsilon>0$ and define
%
\begin{equation}
\phi_{\epsilon}(x)=\frac{1}{\epsilon}\max\{\phi(x),\epsilon\} .
\label{eq:trunc}
\end{equation}
Define $A_0:=\sup_{x}E [\frac{D_{x}f}{\phi(x)} ]$
and $r_\epsilon:=r+ \epsilon A_0$, and assume that $A_0<\infty$.
Then $\phi_{\epsilon}$ is an $r_\epsilon$-sub-eigenfunction for
$\mathcal{L}$.
\end{lemma}

\begin{pf}
Since $\phi_{\epsilon}(x)\geq1$ for every $x$ and
$\frac{\phi_{\epsilon}(f(x))}{\phi_{\epsilon}(x)}\leq
\frac{\phi(f(x))+\epsilon}{\phi(x)}$, we have
%
\begin{equation}
E \biggl[\frac{\phi_{\epsilon}(f(x))}{\phi_{\epsilon}(x)}D_x
f \biggr] \leq
E \biggl[\frac{\phi(f(x))}{\phi(x)}D_{x}f \biggr]+\epsilon E \biggl[\frac{D_xf}{\phi(x)} \biggr]
\leq r+\epsilon A_0 .
\vspace*{6pt}
\end{equation}
\upqed
\end{pf}

The second strategy is to switch to an easier operator, analogously
to switching from one measure to another by the use of a
Radon--Nikodym derivative.

\begin{lemma}\label{lem:switch}
Assume that a
positive linear operator $\mathcal{L}_1$ has the integral
representation $\mathcal{L}_1(\phi)(x) = \int\phi(y)K(x,\mathrm{d}y)$
and let $\mathcal{L}_2(\phi)(x) = \frac{1}{h(x)}\int\phi(y)h(y)K(x,\mathrm{d}y)$,\vspace*{-2pt}
where $h$ is a strictly positive function. Then $\phi$ is an
$r$-sub-eigenfunction for
$\mathcal{L}_1$ if and only if $\frac{\phi}{h}$ is an
$r$-sub-eigenfunction for $\mathcal{L}_2$.
\end{lemma}

\begin{pf}
It is enough to prove one direction only. Let $\phi$ be
an $r$-sub-eigenfunction for
$\mathcal{L}_1$. Then
\[
\mathcal{L}_2 \biggl(\frac{\phi}{h} \biggr)(x) = \frac{1}{h(x)}\int
\frac{\phi(y)}{h(y)}h(y)K(x,\mathrm{d}y) \leq r\frac{\phi(x)}{h(x)} .
\]
\upqed
\end{pf}

In particular, this lemma tells us that if
$r:=\sup_{x}K(x,\chi)<1$, then $1/h$ is an $r$-sub-eigenfunction
for $\mathcal{L}_2$.

\subsection{Example 1: Normal Gibbs sampler}\label{sec3.3}

We shall use the techniques of Section \ref{sec3.2} to find drift functions for
the Gibbs sampler example of Section \ref{sec1}. Recall that, without loss of
generality, we
assume that $K=0$ or $1$ and $\xi=0$. The following proposition gives
three different
drift functions that are valid under different conditions on the
parameters and
the data $Y$. It should be clear that other drift functions are
possible; also, the bounds $r_i$ can be tightened somewhat
at the cost of additional effort and/or more complicated expressions.
For numerical illustrations, see the remarks following the proof of
Proposition \ref{prop-K1bd}.

\begin{proposition} \label{theo:driftK1}
\textup{(i)} For given $K\geq0$, let
\[
A:=\frac{(\alpha+J/2)
(|\bar{Y}|\sqrt{K}+1)(
|\bar{Y}|\sqrt{K}+1/2)
}{\Sigma_{0}^{2}}
\quad\mbox{and}\quad
r_1:= \frac{(|\bar{Y}|\sqrt{K}+1)(|\bar{Y}|\sqrt{K}+{1/2})}{
\alpha+{J/2}-1}.
\]
If $r_1<1$, then for any $\epsilon$ such that
$r_{1,\epsilon}:= r_{1}+\epsilon A<1$, $\phi_{1,\epsilon}(x):=\frac
{1}{\epsilon}\max(\epsilon,\frac{1}{x^2})$ is a drift function with
rate $r_{1,\epsilon}$.
{\smallskipamount=0pt
\begin{longlist}[(iii)]
\item[(ii)] Assume $K=1$. Let $r_2:=(\alpha+\frac{J}{2})\frac
{J^2}{\Sigma_{0}^{2}}(|\bar{Y}|+1)(|\bar{Y}|+\frac{1}{2})$. If $r_2<1$,
then $\phi_2(x)=1$ is a drift function with rate $r_2$.\\
\item[(iii)] Assume $K=1$. Define
%
\begin{equation}\label{eq-defb}
\hat{A} := \frac{(|\bar{Y}|+1)(\alpha+{J/2})J\sqrt{2\uppi}}{2
\Sigma_0^{2}},\qquad
b(x):= \frac{J}{\sqrt{2\uppi}}
\biggl(\frac{2|\bar{Y}|}{(xJ+1)^{3/2}}+\frac{1}{xJ+1} \biggr)
\end{equation}
and
\[
r_3 := \frac{1}{\sqrt{2\uppi}} \biggl(4|\bar{Y}|
\biggl(1-\frac{1}{\sqrt{J(\alpha+J/2)/{\Sigma_0}+1}} \biggr)
+
\log{ \biggl( \frac{ J(\alpha+J/2)}{\Sigma_0}+1 \biggr)}
\biggr).
\]
If $r_3<1$, then for any $\epsilon$ such that
$r_{3,\epsilon}:= r_{3}+\epsilon\hat{A}<1$, the function
$\phi_{3,\epsilon}(x)=\frac{1}{\epsilon}\max(\epsilon,b(x))$ is a drift
function with rate $r_{3,\epsilon}$.
\end{longlist}}
\end{proposition}

\begin{pf}
The idea of the proof is that for each case, we find a
sub-eigenfunction $\phi$ for the operator~$\mathcal{L}$ and, if
necessary, we truncate $\phi$, as in Lemma \ref{lem:trunc}, to obtain a
drift function.

Recall $\mathcal{L}(\phi)(x)=E[\phi(f(x))D_x f]$, where
\[
f(x)=\frac{G}{\Sigma_0+({J}/{2})
(\bar{Y}K/(xJ+K)-Z/{\sqrt{xJ+K}} )^{2}} ,
\]
and $G$ and $Z$ are two independent random variables with
$\Gamma(\alpha+J/2,1)$ and $N(0,1)$ distributions, respectively. We shall
frequently use (without reference) the following two easy calculations
for $G$ and $Z$. First, the definition of the Gamma
distribution implies
that
%
\begin{equation}\label{eq-Gcalc}
E(G^p) = \frac{\Gamma(\alpha+{J}/{2}+p )}{
\Gamma(\alpha+{J}/{2} )}\qquad
\mbox{for }p> - \biggl(\alpha+\frac{J}{2} \biggr).
\end{equation}
Second, for all constants $a,b,c,d$,
the Schwarz inequality and $E(Z^2)=1$ imply
%
\begin{equation}
\label{eq-Zcalc}
E(|a+bZ||c+dZ|) \leq \sqrt{a^2+b^2}\sqrt{c^2+d^2} \leq
(|a|+|b|)(|c|+|d|) .
\end{equation}

(i) The local Lipschitz constant $D_x f$ is equal to the absolute
value of
the derivative $f$ at $x$, so, by direct computation,
%
\begin{equation}
D_{x}f=\frac{GJ^{2} |\bar{Y}K/(xJ+K)-{Z}/{\sqrt{xJ+K}}
| |\bar{Y}K/{(xJ+K)^2}-{Z}/(2(xJ+K)^{{3}/{2}}) |}
{ (\Sigma_0+({J}/{2}) ({\bar{Y}K}/({xJ+K})-{Z}/{\sqrt{xJ+K}}
)^{2} )^{2}} . \label{eq:Lip}%
\end{equation}
Let $k_{x}$ be the joint distribution
of $f(x)$ and $\tilde{D}_x$, where
\[
\tilde{D}_x := \frac{J^{2} |{\bar{Y}K}/{(xJ+K)}-{Z}/{\sqrt{xJ+K}}
| |{\bar{Y}K}/{(xJ+K)^2}-{Z}/({2(xJ+K)^{{3}/{2}}}) |}
{G}
\]
and let $K_{x}(\mathrm{d}c)=x^2 (\int_{0<y<\infty} y k_{x}(\mathrm{d}c,\mathrm{d}y) )$.
Note that $f(x)^2\tilde{D}_x=D_xf$. Therefore,
\[
\mathcal{L}(\phi)(x) = E [\phi(f(x) )D_xf ] =
E [\phi(f(x) )f(x)^2\tilde{D}_x ] =
\frac{1}{h(x)}\int\phi(c)h(c)K_{x}(\mathrm{d}c),
\]
where $h(c)=c^2$. Let $\mathcal{L}_1$ be the operator defined by
$\mathcal{L}_1\phi(x):=\int\phi(c)K_{x}(\mathrm{d}c)$
and let $\mathcal{L}_2=\mathcal{L}$. By Lemma~\ref{lem:switch}, we see
that if $\phi$ is an $r$-sub-eigenfunction for $\mathcal{L}_1$ then
$\frac{\phi}{h}$ is an $r$-sub-eigenfunction for
$\mathcal{L}_2=\mathcal{L}$. We find that
\begin{eqnarray*}
\sup_{x}\int_{0}^{\infty} K_{x}(\mathrm{d}c)&=& \sup_x x^2 E[\tilde{D}_x] \\
&\leq&\sup_{x}\frac{x^2J^2}{(xJ+K)^{2}}\frac{(|\bar{Y}|\sqrt{K}+1)(
|\bar{Y}|\sqrt{K}+{1}/{2})}{\alpha+{J}/{2}-1}\\
&=& r_1 .
\end{eqnarray*}
If $r_1<1$, then $\phi(x)=1$ is an $r_1$-sub-eigenfunction
for $\mathcal{L}_1$ and hence
$\phi_1(x)={x^{-2}}$ is an $r_1$-sub-eigenfunction for
$\mathcal{L}$. Finally, note that for every $x>0$,
\[
E \biggl[\frac{D_xf}{\phi_2(x)} \biggr] =
E \biggl[\frac{(xJ)^{2}}{(xJ+K)^2}
\frac{G |{\bar{Y}K}/{\sqrt{xJ+K}}-Z
| |{\bar{Y}K}/{\sqrt{xJ+K}}-{Z}/{2} |}
{ (\Sigma_0+({J}/{2}) ({\bar{Y}K}/({xJ+K})-{Z}/{\sqrt{xJ+K}}
)^{2} )^{2}} \biggr]
\leq A .
\]
Hence, by Lemma \ref{lem:trunc}, $\phi_{1,\epsilon}$ is a drift function
with growth rate less than $r_1+\epsilon A$.

(ii) When $K=1$, $\sup_x E(D_xf)\leq r_2$. If $r_2<1$ and we let
$\phi_2(x)=1$ $\forall x$, then $\mathcal{L}\phi_2(x)=E(D_xf)\leq r_2\phi
_2(x)$ and
thus $\phi_2(x)$ is a drift function with rate $r_2$.

(iii)
We first derive a more explicit formula for $\mathcal{L}$ and then look for
an operator $\tilde{\mathcal{L}}$ of the form~(\ref{eq:form}) with $n=1$
that dominates $\mathcal{L}$ (as in Corollary \ref{cor:subeigen}).
Note that we can write
\[
\mathcal{L}(\phi)(x)=\int_{0}^{\infty}\phi(c) \biggl(\int_{-\infty}^{\infty}%
\Delta_{x}(z,c)h_{Z,f(x)}(z,c)\,\mathrm{d}z \biggr) \,\mathrm{d}c ,
\]
where $h_{Z,f(x)}$ is
the joint density of $(Z,f(x))$ and
\[
\Delta_{x}(z,c)=
\frac{cJ^2 |{\bar{Y}}/{(xJ+1)}-{z}/{\sqrt{xJ+1}}
| |{\bar{Y}}/{(xJ+1)^2}-{z}/(2(xJ+1)^{3/2}) |}
{\Sigma_0+({J}/{2}) ({\bar{Y}}/({xJ+1})-{z}/{\sqrt{xJ+1}})^{2}}
\]
(observe that $\Delta_{x}(Z,f(x))=D_x f$, by (\ref{eq:Lip})).
To simplify the formulae, let us put
\[
A_x(z)=\frac{\bar{Y}}{(xJ+1)}-\frac{z}{\sqrt{xJ+1}},\qquad
B_x(z)= \biggl|\frac{\bar{Y}}{(xJ+1)^2}-\frac{z}{2(xJ+1)^{{3}/{2}}} \biggr|
\]
and $u_{x}(z) = \Sigma_0+\frac{J}{2}A_x(z)^{2}$.

To find $h_{Z,f(x)}$, we consider the mapping
$T_{x}(z,g)= (z,{g}/{u_{x}(z)} )$.
Note that $T_{x}(Z,G)=(Z,f(x))$. $T_{x}(z,c)$ is one-to-one and
$T_{x}^{-1}(z,c)=(z,
c(u_{x}(z)))$. Let $D$ be the Jacobian of $T^{-1}$. We have
$h_{Z,f(x)}(z,c)=h_{Z,G}(T_x^{-1}(z,c))|\det D|$ and
$|\det D| = u_{x}(z)$; therefore,
\[
h_{Z,f(x)}(z,c) = \frac{1}{\Gamma(\alpha+{J}/{2})\sqrt{2\uppi}}u_{x}(z)%
\mathrm{e}^{-z^{2}/2}(cu_{x}(z))^{\alpha+{J}/{2}-1}%
\mathrm{e}^{-cu_{x}(z)}.
\]
Now,
\begin{eqnarray*}
&&\int_{-\infty}^{\infty}\Delta_{x}(z,c)h_{Z,f(x)}(z,c)\,\mathrm{d}z\\
&&\quad =
\frac{cJ^{2}}{\Gamma(\alpha+{J}/{2})\sqrt{2\uppi}}
\biggl( \int_{z\leq\bar{Y}/{\sqrt{xJ+1}}}\mathrm{e}^{-z^{2}/2}
A_{x}(z)B_{x}(z) (cu_{x}(z) )^{\alpha+{J}/{2}-1}%
\mathrm{e}^{-cu_{x}(z)}\,\mathrm{d}z\\
&&\qquad\hphantom{\frac{cJ^{2}}{\Gamma(\alpha+{J}/{2})\sqrt{2\uppi}}\biggl(}
{} -\int_{z>\bar{Y}/{\sqrt{xJ+1}}}\mathrm{e}^{-z^{2}/2}
A_{x}(z)B_{x}(z)(cu_{x}(z))^{\alpha+{J}/{2}-1}%
\mathrm{e}^{-cu_{x}(z)}\,\mathrm{d}z \biggr).
\end{eqnarray*}
Substituting $u=cu_{x}(z)$ and noting that $\mathrm{d}u=-cJ\frac{1}{\sqrt
{xJ+1}}A_{x}(z)\,\mathrm{d}z$,
we get
\begin{eqnarray*}
&&\int_{-\infty}^{\infty}\Delta_{x}(z,c)h_{Z,f(x)}(z,c)\,\mathrm{d}z\\
&&\quad=\int_{u\geq
c\Sigma_0}\frac{J}{\Gamma(\alpha+{J}/{2})2\sqrt{2\uppi}\sqrt{xJ+1}}
u^{\alpha+{J}/{2}-1}\mathrm{e}^{-u} \\
&&\qquad\hphantom{\int_{u\geq c\Sigma_0}}
{}\times \Biggl[\mathrm{e}^{-({1}/{2})(xJ+1) ((\bar{Y}/(xJ+1))+\sqrt{
({2}/{J})({u}/{c}-\Sigma_0)} )^{2}}%
\Biggl|\frac{\bar{Y}}{xJ+1}-
\sqrt{\frac{2}{J} \biggl(\frac{u}{c}-\Sigma_0 \biggr)} \Biggr|\\
&&\qquad\hphantom{\int_{u\geq c\Sigma_0}\times \Biggl[}
{}+\mathrm{e}^{-({1}/{2})(xJ+1) (({\bar{Y}}/{(xJ+1)})-\sqrt{
({2}/{J})({u}/{c}-\Sigma_0)} )^{2}}%
\Biggl|\frac{\bar{Y}}{xJ+1}
+\sqrt{\frac{2}{J} \biggl(\frac{u}{c}-\Sigma_0 \biggr)} \Biggr| \Biggr]\,\mathrm{d}u.
\end{eqnarray*}
Using the inequality $ |t \mathrm{e}^{-C(A+t)^{2}} |\leq
|A|+\frac{1}{\sqrt{2C}}$\vspace*{-2pt}
(where $A$ and $t$ are real and $C>0$), we bound the term inside the
brackets by $2 (\frac{2|\bar{Y}|}{xJ+1}+\frac{1}{\sqrt{xJ+1}} )$.
Hence, $\mathcal{L}(\phi)(x)\leq b(x)\int_{0}^{\infty}\phi(c)\bar
{H}(c\Sigma_0)\,\mathrm{d}c$,
where $b(x)$ is defined in (\ref{eq-defb})
and $\bar{H}$ is one minus the c.d.f.  of our gamma variable $G$, that
is,
$\bar{H}(x) = \Pr\{G>x\}$.\vadjust{\goodbreak}

Next, we compute $r = \int_{0}^{\infty}b(c)\bar{H}(c\Sigma_0)\,\mathrm{d}c$.
Let $g$ be the density of $G$. Note
%
\begin{eqnarray}
\int_{0}^{\infty}\frac{1}{(cJ+1)^{{3}/{2}}}
\bar{H}(c\Sigma_0)\,\mathrm{d}c&=&\frac{2}{J}\int_{0}^{\infty}
\biggl(1-\frac{1}{\sqrt{xJ/\Sigma_0+1}} \biggr)g(x) \,\mathrm{d}x \nonumber\\
&\leq& \frac{2}{J} \biggl(1-\frac{1}{\sqrt{\int_{0}^{\infty}
({xJ}/{\Sigma_0}+1 )g(x)\,\mathrm{d}x}} \biggr)
\label{eq:Jensen2}\\
&=&\frac{2}{J} \biggl(1-\frac{1}{\sqrt{J(\alpha+{J}/{2})/{\Sigma
_0}+1}} \biggr)
\end{eqnarray}
and
%
\begin{eqnarray}
\int_{0}^{\infty}\frac{1}{cJ+1} \bar{H}(c\Sigma_0) \,\mathrm{d}c&=&
\frac{1}{J}\int_{0}^{\infty}
\log\biggl(\frac{xJ}{\Sigma_0}+1 \biggr)g(x)\,\mathrm{d}x \nonumber\\
&\leq& \frac{1}{J}\log \biggl(\int_{0}^{\infty}
\biggl(\frac{xJ}{\Sigma_0}+1 \biggr)g(x) \,\mathrm{d}x \biggr)
\label{eq:Jensen1}\\
&=&\frac{1}{J}\log{ \biggl(\frac{J(\alpha+{J}/{2})}{\Sigma_0}+1 \biggr)},
\end{eqnarray}
where (\ref{eq:Jensen2}) and (\ref{eq:Jensen1}) follow from Jensen's
inequality. Therefore, $r\leq r_3$. We conclude that $\phi_3$ is an
$r_3$-sub-eigenfunction.

Using (\ref{eq-Zcalc}), we have
\begin{eqnarray*}
E(D_xf) & \leq& \frac{(\alpha+{J}/{2})J^2}{\Sigma_0^2}
E \biggl( \biggl|\frac{\bar{Y}}{xJ+1}-\frac{Z}{\sqrt{xJ+1}} \biggr|
\biggl|\frac{\bar{Y}}{(xJ+1)^2}-\frac{Z}{2(xJ+1)^{3/2}} \biggr| \biggr)
\\
& \leq& \frac{(\alpha+{J}/{2})J^2}{\Sigma_0^2 (xJ+1)}
\biggl( \frac{|\bar{Y}|}{\sqrt{xJ+1}}+1 \biggr)
\biggl(\frac{|\bar{Y}|}{(xJ+1)^{3/2}} +\frac{1}{2(xJ+1)} \biggr) \\
& = & \frac{(\alpha+{J}/{2})J\sqrt{2\uppi}}{2\Sigma_0^2 (xJ+1)}
\biggl( \frac{|\bar{Y}|}{\sqrt{xJ+1}}+1 \biggr) b(x) ,
\end{eqnarray*}
where $b$ is defined in equation (\ref{eq-defb}). Hence, $
\sup_{x}E [{D_{x}f}/{b(x)} ] \leq \hat{A}$.
By Corollary \ref{cor:subeigen} and Lemma~\ref{lem:trunc}, the
function $\phi_{3,\epsilon}$
\cut{$\= \frac{1}{\epsilon}\max\{ \epsilon,
\frac{J}{\sqrt{2\pi}} (\frac{2|\bar{Y}|}{(xJ+1)^{3/2}}
+\frac{1}{xJ+1} ) \}$}
is a drift function with growth rate less than $r_{3,\epsilon}$.
\end{pf}

\begin{proposition}\label{prop-K1bd}
Define $r_i$ and $r_{i,\epsilon}$ as in
Proposition \ref{theo:driftK1}:
\begin{longlist}[(iii)]
\item[(i)] Let $K\geq0$ and assume that $\alpha+J/2>2$. If $r_{1,\epsilon
}<1$, then
for all $x>0$ and all $n\geq1$,
\[
d_{\mathrm{W}}(P_K^{n}(x,\cdot),\pi_K) \leq
\frac{\hat{C}_{1,\epsilon,x}}{1-r_{1,\epsilon}} r_{1,\epsilon}^{n}
,\vadjust{\goodbreak}
\]
where
\begin{eqnarray*}
\hat{C}_{1,\epsilon,x} &=&
\biggl(x+\frac{\alpha+J/2}{\Sigma_0} \biggr) \biggl(\max\biggl\{\frac{1}{\epsilon x^2},1 \biggr\}\\
&&\hphantom{\biggl(x+\frac{\alpha+J/2}{\Sigma_0} \biggr)\biggl(}
{}+\biggl(\Sigma_0^2+\frac{J\Sigma_0}{xJ+K}
\biggl[\frac{(\bar{Y}K)^2}{xJ+K}+1 \biggr] \\
&&\hphantom{\biggl(x+\frac{\alpha+J/2}{\Sigma_0} \biggr) \biggl({}+\biggl(}
{}+\frac{J^2}{4(xJ+K)^2}
\biggl[\frac{(\bar{Y}K)^4}{(xJ+K)^2}+\frac{6(\bar{Y}K)^2}{xJ+K}+3 \biggr]\biggr)
\\
&&\hphantom{\biggl(x+\frac{\alpha+J/2}{\Sigma_0} \biggr) \biggl({}+\ }
{}\times\bigl(\epsilon(\alpha+J/2-1) (\alpha+J/2-2)\bigr)^{-1} \biggr).
\end{eqnarray*}
\item[(ii)] Assume $K=1$. If $r_2<1$, then
for all $x>0$ and all $n\geq1$,
\[
d_{\mathrm{W}}(P_1^{n}(x,\cdot),\pi_1) \leq
\frac{x+({\alpha+{J}/{2}})/{\Sigma_0}}{1-r_2} r_{_2}^{n}.
\]
\item[(iii)] Assume $K=1$. If $r_{3,\epsilon}<1$, then
for all $x>0$ and all $n\geq1$,
\[
d_{\mathrm{W}}(P_1^{n}(x,\cdot),\pi_1) \leq
\frac{\hat{C}_{3,\epsilon,x}}{1-r_{3,\epsilon}} r_{3,\epsilon}^{n} ,
\]
where
\[
\hat{C}_{3,\epsilon,x} :=
\max\biggl\{1, \frac{J(2|\bar{Y}|+1)}{\epsilon \sqrt{2\uppi}} \biggr\}
\biggl(x+\frac{\alpha+J/2}{\Sigma_0} \biggr).
\]
\end{longlist}
\end{proposition}

\begin{pf}
(i) If $r_{1,\epsilon}<1$, then $
d_{\mathrm{W}}(P_K^{n}(x,\cdot),\pi_1) \leq
\frac{C_{1,\epsilon,x}}{1-r_{1,\epsilon}} r_{1,\epsilon}^{n}$,
where
\begin{eqnarray*}
C_{1,\epsilon,x}&=& E \Bigl[|f(x)-x|\sup_{t\in[0,1]}\bigl\{\phi_{1,\epsilon
}\bigl(x+t\bigl(f(x)-x\bigr)\bigr)\bigr\} \Bigr]\\
&\leq&E \biggl[\bigl(f(x)+x\bigr)\max\biggl\{\frac{1}{\epsilon
x^2},\frac{1}{\epsilon f(x)^2},1 \biggr\} \biggr]\\
&\leq&E \biggl[\bigl(f(x)+x\bigr) \biggl(\max\biggl\{\frac{1}{\epsilon
x^2},1 \biggr\} +\frac{1}{\epsilon f(x)^2} \biggr) \biggr]\\
&\leq& E[f(x)+x]E \biggl[\max\biggl\{\frac{1}{\epsilon
x^2},1 \biggr\} +\frac{1}{\epsilon f(x)^2} \biggr],
\end{eqnarray*}
the last line following from the FKG inequality
(see, for example,\ Theorem 3.17 of \cite{MadrasMCMC})
since $1/f(x)^2$ is a decreasing function of the random
variable $f(x)$. From
%
\begin{equation}\label{eq:C1}
x+E(f(x)) \leq x+\frac{\alpha+{J}/{2}}{\Sigma_0}
\end{equation}
and (using equation (\ref{eq-Gcalc}) with $p=-2>-(\alpha+J/2)$)
\begin{eqnarray*}
E (f(x)^{-2} ) &=&
E \biggl[ \biggl(\Sigma_{0}+\frac{J}{2} \biggl(\frac{\bar{Y}K}{xJ+K}-
\frac{Z}{\sqrt{xJ+K}} \biggr)^{2} \biggr)^{2} \biggr] E(G^{-2})  \\
&=&\biggl(\Sigma_{0}^{2}+J\Sigma_0E \biggl[ \biggl(\frac{\bar{Y}K}{xJ+K}-\frac
{Z}{\sqrt{xJ+K}} \biggr)^{2} \biggr]\\
&&\hphantom{\biggl(}
{}+\frac{J^2}{4}E\biggl [ \biggl(\frac{\bar{Y}K}{xJ+K}-\frac
{Z}{\sqrt{xJ+K}} \biggr)^{4} \biggr]\biggr)\\
&&{}\Big/\bigl((\alpha+J/2-1)(\alpha+J/2-2) \bigr) ,
\end{eqnarray*}
and calculation of the expectations in the brackets in the above expression,
we find that $\hat{C}_{1,\epsilon,x}$ is an upper bound for
$C_{1,\epsilon,x}$.

(ii) If $r_2<1$, then $\phi(x)=1$ is a drift function with
rate $r_2$.
Hence,
Theorem \ref{theo:steinsaltz} implies that $d_{\mathrm{W}}(P_1^{n}(x,\cdot),\pi
_1)\leq
\frac{C_{2,x}}{1-r_{2}} r_{2}^{n}$,
and $C_{2,x} = E(|f(x)-x|)\leq x+\frac{\alpha+{J}/{2}}{\Sigma_0}$ by
equation (\ref{eq:C1}).

(iii) If $r_{3, \epsilon}<1$, then $
d_{\mathrm{W}}(P_K^{n}(x,\cdot),\pi_1) \leq
\frac{C_{3,\epsilon,x}}{1-r_{1,\epsilon}} r_{1,\epsilon}^{n}$\vspace*{-2pt}
and $C_{3,\epsilon,x} \leq E[f(x)+x]\sup_y(\phi_{3,\epsilon}(y))
\leq\hat{C}_{3,\epsilon,x}$
because of (\ref{eq:C1}) and the fact that $\sup_y(\phi_{3,\epsilon
}(y))=\max\{1, \frac{J(2|\bar{Y}|+1)}{\epsilon \sqrt{2\uppi}} \}$.
\end{pf}

\begin{remarks*}
(1) The criterion $r_2<1$ is essentially the condition that
$\log\sup_x E(D_xf)<0$. This is similar to the
strong contractivity condition which says that $E(\log\sup_x D_xf)<0$.
Logically, neither condition implies the other.
Each implies the weaker condition
$\sup_{x,y}E(\log[\rho(f(x),f(y))/\rho(x,y)]) < 0$
used in \cite{BE} to prove attractivity (in a more restrictive setting).

(2) In the Bayesian model, as the number of observations $J$ increases,
$\bar{Y}$ and $\Sigma_0/J$ both converge (to $\theta$ and $\sigma^2$,
respectively). Therefore, for large $J$, we expect $r_1$ to be small, but
$r_2$ and $r_3$ to be large.

(3) ($K=1$) To illustrate the calculations in the preceding propositions,
we considered some cases with $5\leq J\leq10$, $\alpha=1$,
$0.5\leq\bar{Y} \leq1.5$ and $5\leq\Sigma_0\leq60$.
As shown in Table \ref{table1}, it is possible for any one of
$r_1$, $r_2$ or $r_3$ to be less than the other two.

\begin{table}[b]
\tablewidth=280pt
\caption{Values of $r_1$, $r_2$ and $r_3$ in three cases of the
Normal Gibbs sampler with $K=1$. Observe that $r_2$ is best in case A,
$r_1$ in case B and $r_3$ in case C. Numbers with ``\ldots'' have
had trailing digits truncated; other numbers are exact}\label{table1}
\begin{tabular*}{280pt}{@{\extracolsep{\fill}}llllllll@{}}
\hline
Case & $J$ & $\alpha$ & $\bar{Y}$ &
$\Sigma_0$ & $r_1$ & $r_2$ & $r_3$ \\
\hline
A & 10 & 1 & 1.5 & 60 & 1 & $5/6$ & 0.97\ldots\\
B & \hphantom{0}5 & 1 & 0.5 & \hphantom{0}5 & 0.6 & 5.25 & 1.02\ldots\\
C & \hphantom{0}5 & 1 & 1 & 12 & 1.2& 1.82\ldots& 0.9368\ldots\\
\hline
\end{tabular*}
\end{table}

(a) In case A, we have $r_2=5/6$ and $\hat{C}_{2,x}=x+0.1$. Hence,
for $x=1$, we have
\[
d_{\mathrm{W}}(P_1^{n}(1,\cdot),\pi_1) \leq 6.6*(5/6)^n
\qquad\mbox{for $n\geq1$ in case A.}
\]
In particular, $d_{\mathrm{W}}(P_1^{n}(1,\cdot),\pi_1)<0.01$ for $n\geq36$ in
case A.

(b) For case B, we have $r_1=0.6$ and $A=0.21$. We want to have
$r_{1,\epsilon}<1$, where $r_{1,\epsilon}=0.6+0.21\epsilon$.
Suppose we choose
$\epsilon=0.5$. Then $r_{1,\epsilon}=0.705$ and
$\hat{C}_{1,\epsilon,x}< (16+\max\{1,2x^{-2}\})(x+0.7)$ for all $x>0$.
For $x=1$, we obtain
\[
d_{\mathrm{W}}(P_1^{n}(1,\cdot),\pi_1) \leq 104*0.705^n
\qquad\mbox{for $n\geq1$ in case B.}
\]
In particular, $d_{\mathrm{W}}(P_1^{n}(1,\cdot),\pi_1)<0.01$ for $n\geq27$ in
case B.

(c) In case C, we have $r_3<0.9369$ and $\hat{A}<0.305$. Choosing
$\epsilon=0.01$ gives $r_{3,\epsilon}< 0.94$ and
$\hat{C}_{3,\epsilon,x}<599(x+0.3)$. For $x=1$, we obtain
\[
d_{\mathrm{W}}(P_1^{n}(1,\cdot),\pi_1) \leq 12980*0.94^n
\qquad\mbox{for $n\geq1$ in case C.}
\]
Therefore, $d_{\mathrm{W}}(P_1^{n}(1,\cdot),\pi_1)<0.01$ for $n\geq228$ in case C.

(4) ($K=0$) Consider the three cases of Table \ref{table1}, but now
using the prior distribution with \mbox{$K=0$}.
Table \ref{table2} gives
the calculations of Propositions \ref{theo:driftK1}(i) and
\ref{prop-K1bd}(i) (note that $r_1= 1/[2\alpha+J-2]$);
the last column is the bound on the Wasserstein
distance from equilibrium after $n$ iterations, started from $x=1$. We
find that $d_{\mathrm{W}}(P_0^{n}(1,\cdot),\pi_0)<0.01$ for $n\geq5$
in case A and for $n\geq6$ in cases B and C.

\begin{table}
\caption{Values of expressions from Propositions
\protect\ref{theo:driftK1}(i) and \protect\ref{prop-K1bd}(i) for the
Normal Gibbs sampler with $K=0$, for the cases given in Table
\protect\ref{table1}. The values of $\epsilon$ were chosen somewhat
arbitrarily. We use $x=1$ in all cases. Numbers with ``\ldots'' have
had trailing digits truncated; other numbers are exact}\label{table2}
\begin{tabular*}{\textwidth}{@{\extracolsep{\fill}}lllllll@{}}
\hline
Case & $r_1$ & $A$ & $\epsilon$ & $r_{1,\epsilon}$
& $\hat{C}_{1,\epsilon,x}$ & $d_{\mathrm{W}}(P_0^{n}(1,\cdot),\pi_0) \leq$ \\
\hline
A & 0.1 & $1/1200$ & 1 & 0.1008\ldots& 202.4\ldots& $226*(0.101)^n$ \\
B & 0.2 & 0.07 & 0.5 & 0.235 & \hphantom{0}31.28\ldots& \hphantom{0}$40.9*(0.235)^n$ \\
C & 0.2 & 0.012\ldots& 1 & 0.212\ldots& \hphantom{0}55.28\ldots&
\hphantom{0}$70.3*(0.213)^n$ \\
\hline
\end{tabular*}
\end{table}

\end{remarks*}

\section{From Wasserstein distance to total variation distance}
\label{sec-WTV}
\subsection{One-shot coupling}\label{sec4.1}
In this section, we present Theorem \ref{prop-WTVgen}, our main tool
for converting Wasserstein convergence rates to total variation
convergence rates.
Various methods of coupling
have been used for proving convergence in TV distance
\cite{Diac,Jerrum,RobRos04}.
Although not explicit in the final formulation,
the idea behind this theorem is a certain kind of coupling
method, called \textit{one-shot coupling}, which has been successfully
applied to iterated function systems by Roberts and
Rosenthal \cite{RobRos02} (see also \cite{BesRob,Huber}).
We describe this method now.

We shall consider two copies of a Markov chain, running
simultaneously. Let $S_0$ and $\tilde{S}_0$ be two initial values
for this chain (possibly random with some joint distribution). Let
$\{f_t\}$ be a sequence of i.i.d.\ random maps that defines this
Markov chain. Define
\[
S_t = f_t(S_{t-1}) \quad\mbox{and}\quad
\tilde{S}_t = f_t(\tilde{S}_{t-1})
\qquad\mbox{for $t=1,\ldots,n-1$.}
\]
That is, we use the same realization of the functions $f_t$ on\vspace*{-2pt}
both copies of the chains, up to time $n-1$.
Suppose, at time $n$, we can find two
copies $\hat{f}_{n}$ and $\hat{\hat{f}_n\hphantom{\,}}$ of $f_n$, that are independent
from everything earlier (but not independent of each other),\vspace*{-2pt} such that,
with high probability, we have
$\hat{f}_n(S_{n-1})=\hat{\hat{f}_n\hphantom{\,}}(\tilde{S}_{n-1})$.
(The name ``one-shot coupling'' refers to the fact that we only try to
coalesce the two copies of the chain at the single time $n$.)
By the representation (\ref{eq-TVcouprep}), this would imply
that $S_{n}$ and $\tilde{S}_n$ are close to each other in TV distance.\vspace*{-2pt}
Two conditions help us to find such $\hat{f}_{n}$ and
$\hat{\hat{f}_n\hphantom{\,}}\!$: first, $S_{n-1}$ and $\tilde{S}_{n-1}$ need to be
reasonably close; second, the density functions of
the two random variables $f_t(x)$ and
$f_t(y)$ need to have a large overlap when $x$ and $y$ are close.
Theorem \ref{prop-WTVgen} is a precise refinement of this argument.

In what follows, let $(\chi,\rho)$ be a complete separable metric
space and let $P$ be a transition probability operator on the state space
$\chi$. Assume that $P$ has a density $p$ with respect to some
reference measure $\lambda$ (that is, $P(x,\mathrm{d}z)=p(x,z) \lambda(\mathrm{d}z)$).
Let $\mu$ be any probability distribution on~${\chi}$ and let $\pi$
be a stationary probability distribution for $P$.

\begin{theorem}
\label{prop-WTVgen}
\textup{(a)} Assume that there is a constant $A$ such that
%
\begin{equation}
\label{eq.QTVprop0}
\int_{\chi}|p(x,z) - p(y,z)| \lambda(\mathrm{d}z) \leq A \rho(x,y)
\qquad\mbox{for all $x,y\in{\chi}$}.
\end{equation}
Then
\[
d_{\mathrm{TV}}(\mu P^n,\pi) \leq \frac{A}{2} d_{\mathrm{W}}(\mu P^{n-1},\pi)
\qquad\mbox{for all $n\geq1$}.
\]
\begin{enumerate}[(a)]
\item[(b)] Assume the following conditions hold:
\begin{longlist}[(ii)]
\item[(i)] there exists a function $h>0$ on $\chi$ such that
%
\begin{equation}
\label{eq.QTVprop1}
\int_{\chi}|p(x,z) - p(y,z)| \lambda(\mathrm{d}z) \leq
\frac{\rho(x,y)}{\max\{h(x),h(y)\}}\qquad
\mbox{for all $x,y\in{\chi}$; }
\end{equation}

\item[(ii)] there exist positive constants $B$, $q$ and $\epsilon_0$ such that
%
\begin{equation}
\label{eq.QTVprop2}
\pi\bigl(\{y \dvtx h(y)<\epsilon\} \bigr) \leq B\epsilon^q
\qquad\mbox{for all $\epsilon$ in $(0,\epsilon_0)$}.
\end{equation}
%
Let $\tilde{C} = (2q)^{-q/(1+q)}
\max\{(q+1)B^{1/(1+q)},(B^q\epsilon_0)^{-1/(1+q)}\}$. Then
%
\begin{equation}
\label{eq-LTVWrel}
d_{\mathrm{TV}}(\mu P^n,\pi) \leq \tilde{C}
[ d_{\mathrm{W}}(\mu P^{n-1},\pi) ]^{q/(1+q)}
\qquad\mbox{for all $n\geq1$}.
\end{equation}
\end{longlist}
\end{enumerate}
\end{theorem}

\begin{remarks*}
(1) If we also know $\limsup_{n\rightarrow\infty} [ d_{\mathrm{W}}(\mu P^{n},\pi)]^{1/n}
\leq \rho<1$,
then the conditions of Theorem \ref{prop-WTVgen}(b) imply that
$\limsup_{n\rightarrow\infty} [d_{\mathrm{TV}}(\mu P^{n},\pi)]^{1/n}
\leq \rho^{q/(1+q)}$.
{\smallskipamount=0pt
\begin{longlist}[(3)]
\item[(2)] Observe that condition (\ref{eq.QTVprop0}) should not
be expected to hold uniformly for $x$ and $y$ near 0 in
the random logistic model. Indeed, as $x$ decreases to 0, the
density of $f_t(x)$ becomes more and more peaked near 0. Essentially,
this is because 0 is a fixed point of the continuous random function $f_t$.
The same thing happens in the Gibbs sampler example when $K$ is 0.

\item[(3)] Lemma \ref{lem-gABC} will be useful in obtaining bounds of the
form (\ref{eq.QTVprop0}) or (\ref{eq.QTVprop1}).
\end{longlist}}
\end{remarks*}

Our first step in proving the above theorem is the following
calculation.

\begin{lemma}
\label{lem-wasser}
Let $\eta$ and $\nu$ be probability measures on $\chi$. Let $\Psi$
be a probability measure in $\operatorname{Joint}(\eta,\nu)$. Then
%
\begin{equation}
\label{eq-tvphi}
d_{\mathrm{TV}}(\eta P,\nu P) \leq \frac{1}{2} \int_x \int_y \int_z
|p(x,z)-p(y,z)| \lambda(\mathrm{d}z) \Psi(\mathrm{d}x,\mathrm{d}y) .
\end{equation}
\end{lemma}

\begin{pf}
Since $(\eta P)(\mathrm{d}z) = (\int_x\eta(\mathrm{d}x) p(x,z) ) \lambda(\mathrm{d}z)$
and similarly for $\nu P$, we apply equation (\ref{eq-TVpdf}) to obtain
\begin{eqnarray*}
d_{\mathrm{TV}}(\eta P,\nu P) & = & \frac{1}{2} \int_z
\biggl| \int_x\eta(\mathrm{d}x) p(x,z) - \int_y\nu(\mathrm{d}y) p(y,z) \biggr| \lambda(\mathrm{d}z)
\\
& = & \frac{1}{2} \int_z \biggl|
\int_x\int_y p(x,z) \Psi(\mathrm{d}x,\mathrm{d}y) - \int_x\int_y p(y,z) \Psi(\mathrm{d}x,\mathrm{d}y)
\biggr| \lambda(\mathrm{d}z)
\\
& \leq& \frac{1}{2} \int\!\!\int\!\!\int |p(x,z) - p(y,z) |
\lambda(\mathrm{d}z) \Psi(\mathrm{d}x,\mathrm{d}y) .
\end{eqnarray*}
\upqed
\end{pf}

\begin{pf*}{Proof of Theorem \ref{prop-WTVgen}}
We shall apply Lemma \ref{lem-wasser} with $\eta=\mu P^{n-1}$ and $\nu
=\pi$
($=\pi P$). Recall from Section \ref{sec-metrics} that there is a
probability measure $\Psi\equiv\Psi_{\eta,\nu}$
in $\operatorname{Joint}(\eta,\nu)$ such that
$ d_{\mathrm{W}}(\eta,\nu) = \int_x\int_y \rho(x,y) \Psi(\mathrm{d}x,\mathrm{d}y)$.
The proof of part (a) follows immediately.

For part (b), let $\epsilon>0$.
Observe that the left-hand side of equation (\ref{eq.QTVprop1}) is never
greater than 2. Lemma \ref{lem-wasser} and the assumption (\ref
{eq.QTVprop1}) then imply that
%
\begin{equation}
\label{eq-LTVAB1}
d_{\mathrm{TV}}(\eta P,\nu P) \leq I_A + I_B ,
\end{equation}
where
\[
I_A = \frac{1}{2}
\int\int_{\{x,y \dvtx \max\{h(x),h(y)\}\geq\epsilon\}}
\frac{\rho(x,y)}{\max\{h(x),h(y)\}} \Psi(\mathrm{d}x,\mathrm{d}y)
\]
and
\[
I_B = \int\int_{\{x,y \dvtx \max\{h(x),h(y)\}<\epsilon\}}
1 \Psi(\mathrm{d}x,\mathrm{d}y).
\]
Note that
\[
I_A \leq \frac{1}{2}
\int\int_{\{x,y\dvtx \max\{h(x),h(y)\}\geq\epsilon\}}
\frac{\rho(x,y)}{\epsilon} \Psi(\mathrm{d}x,\mathrm{d}y)
\leq \frac{d_{\mathrm{W}}(\mu,\nu)}{2 \epsilon}
\]
and $I_B \leq \pi(\{ y \dvtx h(y)< \epsilon \}) $.
Combining these bounds with the assumption (\ref{eq.QTVprop2}) tells us that
\[
d_{\mathrm{TV}}(\mu P^n,\pi) \leq \frac{d_{\mathrm{W}}(\mu P^{n-1},\pi)}{2 \epsilon}
+ B \epsilon^q
\qquad\mbox{for all }\epsilon\in(0,\epsilon_0).
\]

Let $A_n=d_{\mathrm{W}}(\mu P^{n-1},\pi)$ and consider the function
$G_n(\epsilon)=A_n/(2\epsilon)+B\epsilon^q$. Simple calculus shows
that $G_n$ is minimized at $ \epsilon_n := ( \frac{A_n}{2Bq} )^{1/(1+q)}$\vspace*{-3pt}
and the minimum value of the function is
$ G_n(\epsilon_n) = C_{Bq}A_n^{q/(1+q)}$, where
$C_{Bq} = (q+1)(Bq^{-q}2^{-q})^{1/(1+q)}$.
Let $\alpha_0 =2Bq \epsilon_0^{1+q}$. If $A_n<\alpha_0$, then
$\epsilon_n<\epsilon_0$, so $d_{\mathrm{TV}}(\mu P^n,\pi) \leq G_n(\epsilon_n)$.
If $A_n\geq\alpha_0$, then, trivially, $d_{\mathrm{TV}}(\mu P^n,\pi) \leq1\leq
\alpha_0^{-q/(1+q)}A_n^{q/(1+q)}$. Thus equation (\ref{eq-LTVWrel})
holds with $\tilde{C}=\max\{C_{Bq},\alpha_0^{-q/(1+q)}\}$.
\end{pf*}

\subsection{Example 1: Normal Gibbs sampler}\label{sec4.2}

We return to the Gibbs sampler example described in Section 1.
Recall that we write
$P_K$, $p_K$ and $\pi_K$ to denote the corresponding
transition kernel, density and stationary distribution, where
$K\in\{0,1\}$, without loss of generality.

\begin{proposition}
\label{prop-GSTVW}
Let $\mu$ be
an arbitrary initial probability distribution on $(0,\infty)$. Then
%
\begin{equation}
\label{eq-GSTVW1}
d_{\mathrm{TV}}(\mu P_1^n,\pi_1) \leq
\frac{J}{2} \biggl(1+\frac{|\bar{Y}|}{\sqrt{2\uppi}}
\biggr) d_{\mathrm{W}}(\mu P_1^{n-1},\pi_1)\qquad
\mbox{for $n=1,2,\ldots$}
\end{equation}
and
%
\begin{equation}
\label{eq-GSTVW0}
d_{\mathrm{TV}}(\mu P_0^n,\pi_0) \leq \tilde{C} d_{\mathrm{W}}(\mu P_0^{n-1},\pi_0)^w
\qquad\mbox{for $n =1,2,\ldots,$}
\end{equation}
where
\[
w = \frac{2\alpha+J-1}{2\alpha+J+1}
\]
and
\[
\tilde{C} = \biggl(\alpha+\frac{J+1}{2} \biggr)
\mathrm{e}^{(1-w)\Sigma_0}(2\alpha+J-1)^{-w}.
\]
\end{proposition}

Before proceeding, let us revisit the numerical examples of
Table \ref{table1}, as discussed in the remarks following
Proposition \ref{prop-K1bd}.

(a) ($K=1$) If $d_{\mathrm{W}}(\mu P_1^n,\pi_1)\leq Q S^n$ for
some constants $Q$ and $S$, then $d_{\mathrm{TV}}(\mu P_1^n,\pi_1)\leq
\frac{J}{2}(1+|\bar{Y}|/\sqrt{2\uppi})(Q/S)S^n$.
Thus, for the case where $\mu$ is the point mass at $x=1$,
we obtain the following upper bounds on $d_{\mathrm{TV}}(\mu P_1^n,\pi_1)$:
$63.3(5/6)^n$ in case A, $443(0.705)^n$ in case B and
$48\mbox{,}294(0.94)^n$ in case C. Hence, the total variation distance
to equilibrium is less then $0.01$ when $n\geq49$ in case A,
when $n\geq31$ in case B and when $n\geq249$ in case~C.

(b) ($K=0$) We have $w=11/13$ in case A and $w=3/4$ in cases B and C.
Numerical values for $\tilde{C}$ (rounded up) are 8722 in case A,
$3.642$ in case B and $20.96$ in case C. If we know that
$d_{\mathrm{W}}(\mu P_0^n,\pi_0)\leq Q S^n$, then we obtain
$d_{\mathrm{TV}}(\mu P_0^n,\pi_0)\leq\tilde{C}(Q/S)^w(S^{w})^{n}$.
Thus, for the case where $\mu$ is the point mass at $x=1$,
we obtain the following upper bounds on $d_{\mathrm{TV}}(\mu P_0^n,\pi_0)$:
$5\mbox{,}958\mbox{,}000(0.144)^n$ in case A,
$174.6(0.338)^n$ in case B and
$1624(0.314)^n$ in case C.
Therefore, $d_{\mathrm{TV}}(P^n_0(1,\cdot),\pi_0)<0.01$ for
$n\geq11$ in cases A and C, and for $n\geq10$ in case~B.

Logically, the proof of this proposition belongs at the end of this
section since it relies on several lemmas that have not yet been proven.
However, we shall present the proof now since it serves as a guide
for what is to come.

\begin{pf*}{Proof of Proposition \ref{prop-GSTVW}}
Equation (\ref{eq-GSTVW1}) follows from Theorem \ref{prop-WTVgen}(a)
and Lemma \ref{lem-ppbound1} below.
Equation (\ref{eq-GSTVW0}) follows from Theorem \ref{prop-WTVgen}(b)
and Lemmas \ref{lem-ppbound0} and \ref{lem-mubound} below.
In Theorem \ref{prop-WTVgen}(b), we use
$q=\alpha+(J-1)/2$, $B=\mathrm{e}^{\Sigma_0}$ and $\epsilon_0=1$ (all
courtesy of Lemma \ref{lem-mubound}), and it
is not hard to check that, in the definition of $\tilde{C}$,
the first term inside the `$\max$' exceeds the second.
\end{pf*}

The proof of Lemma \ref{lem-mubound} relies on our knowledge of
the explicit form of the equilibrium distribution (which is
known in many MCMC problems).
The proofs of Lemmas \ref{lem-ppbound1} and \ref{lem-ppbound0} rely
heavily on
Lemma \ref{lem-gABC}, together with the following technical lemma.

\begin{lemma}\label{lem-techA}
Let $Z$ be a standard Normal random variable:
\begin{longlist}[(a)]
\item[(a)] Let $a$ and $b$ be positive constants. Then
$ d_{\mathrm{TV}} ( \frac{Z}{\sqrt{a}}, \frac{Z}{\sqrt{b}} )
\leq |a-b|/\max\{a,b\}$.

\item[(b)] Let $t$ be a real constant. Then
$ d_{\mathrm{TV}}(Z,Z+t) \leq |t|/\sqrt{2\uppi} $.
\end{longlist}
\end{lemma}

\begin{pf}
For positive $x$, let
$\phi_x(\cdot)$ be the probability density function of $Z/\sqrt{x}$,
that is, $
\phi_x(t) = \sqrt{\frac{x}{2\uppi}} \mathrm{e}^{-xt^2/2}$
$(t\in\mathbb{R})$.

(a) Without loss of generality, assume that $0<a<b$.
Using equation (\ref{eq-TVpdf1}) and $\mathrm{e}^{-at^2/2}>\mathrm{e}^{-bt^2/2}$, we obtain
\begin{eqnarray*}
d_{\mathrm{TV}} \biggl( \frac{Z}{\sqrt{a}}, \frac{Z}{\sqrt{b}} \biggr) & = &
\int_{t \dvtx \phi_b(t)>\phi_a(t)}
\Biggl(\sqrt{\frac{b}{2\uppi}} \mathrm{e}^{-bt^2/2} -
\sqrt{\frac{a}{2\uppi}} \mathrm{e}^{-at^2/2} \Biggr)\, \mathrm{d}t
\\
& < &
\int_{t \dvtx \phi_b(t)>\phi_a(t)}
\Biggl( \sqrt{\frac{b}{2\uppi}} \mathrm{e}^{-bt^2/2} -
\sqrt{\frac{a}{2\uppi}} \mathrm{e}^{-bt^2/2} \Biggr) \,\mathrm{d}t
\\
& \leq& \bigl(\sqrt{b} - \sqrt{a}\bigr)
\int_{-\infty}^{\infty}
\frac{1}{\sqrt{2\uppi}} \mathrm{e}^{-bt^2/2} \,\mathrm{d}t
\\
& = & \bigl(\sqrt{b} - \sqrt{a}\bigr)\frac{1}{\sqrt{b}}
\leq \frac{|b-a|}{b} .
\end{eqnarray*}
Since $b=\max\{a,b\}$, this proves part (a).

(b) Let $\phi=\phi_1$, the probability density function of $Z$.
Then $\phi(\cdot-t)$ is the probability density function of $Z+t$.
By symmetry, we can assume that $t>0$. Observe that
the function $\min\{\phi(u),\phi(u-t)\}$ equals
$\phi(u)$ for $u\geq t/2$ and is symmetric (with respect to $u$)
about\vadjust{\goodbreak}
$u=t/2$. Using this observation with equation (\ref{eq-TVpdf2}) shows that
\begin{eqnarray*}
d_{\mathrm{TV}}(Z,Z+t) &=& 1-\int_{-\infty}^{\infty}\min\{\phi(u),\phi(u-t)\}
\,\mathrm{d}u \\
&=& 1-2\int_{t/2}^{\infty} \phi(u) \,\mathrm{d}u =
\int_{-t/2}^{t/2} \phi(u) \,\mathrm{d}u \leq \frac{t}{\sqrt{2\uppi}} ,
\end{eqnarray*}
where we have used the bound $\phi(u)\leq1/\sqrt{2\uppi}$ for all $u$.
This proves part (b).
\end{pf}

\begin{lemma}[$\bolds{(K=1)}$]
\label{lem-ppbound1}
For all positive $x$ and $y$,
\[
d_{\mathrm{TV}}(P_1(x,\cdot),P_1(y,\cdot))
\leq J|x-y| \bigl(1+ |\bar{Y}|/\sqrt{2\uppi} \bigr).
\]
\end{lemma}

\begin{pf}
For given $s>0$,
$p_1(s,\cdot)$ is the probability density function of (\ref{GS.Afdef})
with $K=1$ and $\xi=0$. Therefore,
Lemma \ref{lem-gABC} implies that
\[
d_{\mathrm{TV}}(P_1(x,\cdot),P_1(y,\cdot))
\leq d_{\mathrm{TV}} \biggl( \frac{Z}{\sqrt{a}}-\frac{\bar{Y}}{a} ,
\frac{Z}{\sqrt{b}}-\frac{\bar{Y}}{b} \biggr),
\]
where $a=xJ+1$, $b=yJ+1$ and $Z\sim N(0,1)$. We then have
\begin{eqnarray*}
d_{\mathrm{TV}}(P_1(x,\cdot),P_1(y,\cdot)) & = &
d_{\mathrm{TV}} \biggl( \frac{Z}{\sqrt{a}} , \frac{Z}{\sqrt{b}} +\bar{Y}
\biggl[\frac{1}{a}-\frac{1}{b} \biggr] \biggr) \\
& \leq& d_{\mathrm{TV}} \biggl( \frac{Z}{\sqrt{a}} , \frac{Z}{\sqrt{b}} \biggr)
+ d_{\mathrm{TV}} \biggl( \frac{Z}{\sqrt{b}} ,
\frac{Z}{\sqrt{b}} +\bar{Y} \biggl[ \frac{b-a}{ab} \biggr] \biggr) \\
& = & d_{\mathrm{TV}} \biggl( \frac{Z}{\sqrt{a}} , \frac{Z}{\sqrt{b}} \biggr)
+ d_{\mathrm{TV}} \biggl( Z , Z+ \bar{Y} \biggl[ \frac{b-a}{a\sqrt{b}} \biggr]
\biggr) \\
& \leq& \frac{|b-a|}{\max\{a,b\}} +
\frac{|\bar{Y}||b-a|}{\sqrt{2\uppi} a\sqrt{b}}
\qquad\mbox{(by Lemma \ref{lem-techA})}.
\end{eqnarray*}
Finally, since $|a-b|=J|x-y|$ and $a,b\geq1$, the lemma follows.
\end{pf}

\begin{lemma}[$\bolds{(K=0)}$]
\label{lem-ppbound0}
For all positive $x$ and $y$,
\[
d_{\mathrm{TV}}(P_0(x,\cdot),P_0(y,\cdot)) =
\frac{1}{2}\int_0^{\infty}|p_0(x,z)-p_0(y,z)| \,\mathrm{d}z
\leq \frac{|x-y|}{\max\{x,y\}} .
\]
\end{lemma}

\begin{pf}
The equality in the lemma comes from equation (\ref{eq-TVpdf}).
Recall from equation (\ref{GS.Afdef0})
that $p_0(x,\cdot)$ is the probability density function of
$G/(\Sigma_0+\frac{1}{2}[Z/\sqrt{x}]^2)$, where $G$ has a particular Gamma
distribution and $Z$ has the standard Normal distribution. Therefore,
Lemma \ref{lem-gABC} implies that
$ d_{\mathrm{TV}}(P_0(x,\cdot),P_0(y,\cdot))
\leq d_{\mathrm{TV}} ( \frac{Z}{\sqrt{x}},
\frac{Z}{\sqrt{y}} ) $
and Lemma \ref{lem-techA}(a) completes the proof.
\end{pf}

\begin{lemma}[$\bolds{(K=0)}$]
\label{lem-mubound}
$\pi_0 ( [0,\epsilon] ) \leq \mathrm{e}^{\Sigma_0}
\epsilon^{\alpha+(J-1)/2}$ for all $\epsilon$ in $(0,1]$.
\end{lemma}

\begin{pf}
The density $\pi_0(s)$ is the integral over $\theta$
of the posterior density $p(\theta,s|Y)$, which is given by
equation (\ref{GS-post}) with $K=0$.
Using equation (\ref{GS.SS}), we see that
\[
p(\theta,s|Y) =
\frac{1}{\zeta}
s^{\alpha-1+J/2}\exp[-sJ(\bar{Y}-\theta)^2/2] \mathrm{e}^{-s\Sigma_0}
\qquad\mbox{for $s > 0$ and $\theta\in\mathbb{R}$},
\]
where $\zeta=\zeta(\alpha,J,\Sigma_0,Y)$ is the normalizing constant.
Truncating the double integral that defines $\zeta$ shows that
\[
\zeta \geq \mathrm{e}^{-\Sigma_0} \int_0^1 \int_{-\infty}^{\infty}
s^{\alpha-1+J/2} \exp[-sJ(\bar{Y}-\theta)^2/2] \,\mathrm{d}\theta \,\mathrm{d}s .
\]
Therefore, for $\epsilon$ in $(0,1]$,
\begin{eqnarray*}
\pi_0 ( [0,\epsilon] ) & \leq& \frac{1}{\zeta}
\int_0^{\epsilon} \int_{-\infty}^{\infty}
s^{\alpha-1+J/2} \exp[-sJ(\bar{Y}-\theta)^2/2] \,\mathrm{d}\theta \,\mathrm{d}s \\
& \leq& \mathrm{e}^{\Sigma_0} \frac{\int_0^{\epsilon} s^{\alpha-3/2+J/2} \,\mathrm{d}s}{
\int_0^1 s^{\alpha-3/2+J/2} \,\mathrm{d}s} =
\mathrm{e}^{\Sigma_0}\epsilon^{\alpha-1/2+J/2} ,
\end{eqnarray*}
using
$\int_{-\infty}^{\infty} \exp[-sJ(\bar{Y}-\theta)^2/2] \,\mathrm{d}\theta=
(2\uppi Js)^{-1/2}$ in the second inequality.
\end{pf}

\begin{remark*}
Although we did not do it, one can compute $\zeta$ exactly when $K=0$.
In most practical MCMC applications, the normalizing
constant is hard to evaluate or even estimate -- which is one reason
that people use MCMC instead of numerical analysis.
In general, finding constants $B$ and
$\epsilon_0$ for equation (\ref{eq.QTVprop2}) can be hard. The above proof
suggests one way to approach the challenge.
\end{remark*}

\subsection{Example 2: Random logistic maps}\label{sec4.3}
Recall that we are considering i.i.d.\ random maps $f_1,f_2,\ldots$ on $[0,1]$
defined by
\[
f_i(x) = 4 B_i x(1-x),
\]
where $B_i\sim\operatorname{Beta}(a+\frac{1}{2},a-\frac{1}{2})$ [$a>\frac{1}{2}$],
and that
the $\operatorname{Beta}(a,a)$ distribution is the unique stationary distribution for
the iterated function system.

In this subsection, we prove Theorem \ref{thm-logconv}.
The proof of this theorem is similar to the proof of the `$K=1$' part of
Proposition \ref{prop-GSTVW}.

We begin with some notation.
Let $b(t)$ be the density of the $B_i$'s, that is,
\[
b(t) = \cases{
K_a t^{a-1/2}(1-t)^{a-3/2}  &\quad for $0\leq t\leq1$, \cr
0  &\quad otherwise,
}
\]
where $K_a = \Gamma(2a)/\Gamma(a+\frac{1}{2})\Gamma(a-\frac{1}{2})$.
Let
\[
Q(x) = 4x(1-x) \qquad\mbox{for }0\leq x \leq1.
\]
Observe that $0\leq Q(x)\leq1$ for $0\leq x \leq1$.
For a given $x\in(0,1)$, let $b_x(\cdot)$ be the probability
density function of $B_iQ(x)$, that is,
\[
b_x(z) = \cases{
\dfrac{1}{Q(x)} b \biggl( \dfrac{z}{Q(x)} \biggr) &\quad
for $0\leq z \leq Q(x)$, \cr
0  & \quad otherwise.
}
\]
Next, let $p(x,z)$ denote the transition density of the Markov chain
corresponding to the iterated logistic maps. We then have
%
\begin{equation}
\label{eq-logpdf}
p(x,z) = b_x(z) \qquad\mbox{for }x,z\in[0,1].
\end{equation}

\begin{lemma}
\label{lem-logwass}
For the iterated logistic maps with $a> 1/2$, we have
\[
\frac{1}{2}\int_0^1|p(x,z)-p(y,z)| \,\mathrm{d}z \leq
\frac{8a|y-x|}{\max\{Q(x),Q(y)\} }
\qquad\mbox{for }x,y\in(0,1).
\]
\end{lemma}

\begin{pf}
Without loss of generality, assume that $0<Q(x)\leq Q(y)$.
By equation (\ref{eq-logpdf}),
Proposition~\ref{prop-Wbasic}
and some calculation similar to that which was involved in the proof of
Lemma \ref{lem-techA}, we have
%
\begin{eqnarray}\label{eq.Qratio}
\nonumber
&&\frac{1}{2}\int_0^1|p(x,z)-p(y,z)| \,\mathrm{d}z \\
&&\quad=
\int_{\{z \dvtx b_x(z)>b_y(z)\}} \bigl(b_x(z)-b_y(z) \bigr) \,\mathrm{d}z \nonumber\\
\nonumber
&&\quad= \int_{\{z \dvtx b_x(z)>b_y(z)\}} K_a
\biggl( \frac{z^{a-1/2}(Q(x)-z)^{a-3/2}}{Q(x)^{2a-1}} -
\frac{z^{a-1/2}(Q(y)-z)^{a-3/2}}{Q(y)^{2a-1}} \biggr) \,\mathrm{d}z
\nonumber\\
&&\quad <  \int_{\{z \dvtx b_x(z)>b_y(z)\}} K_a
z^{a-1/2} \biggl( \frac{(Q(x)-z)^{a-3/2}}{Q(x)^{2a-1}} -
\frac{(Q(x)-z)^{a-3/2}}{Q(y)^{2a-1}} \biggr) \,\mathrm{d}z
\nonumber\\[-8pt]\\[-8pt]
& & \qquad\quad\mbox{(since $Q(y)-z\geq Q(x)-z\geq0$)}\nonumber \\
\nonumber
&&\quad= \biggl( \frac{1}{Q(x)^{2a-1}} - \frac{1}{Q(y)^{2a-1}} \biggr)
\int_{\{z \dvtx b_x(z)>b_y(z)\}} K_a z^{a-1/2}\bigl(Q(x)-z\bigr)^{a-3/2} \,\mathrm{d}z \\
\nonumber
&&\quad \leq \biggl(1- \biggl( \frac{Q(x)}{Q(y)} \biggr)^{2a-1} \biggr)
\int_0^{Q(x)} \frac{ K_a z^{a-1/2}(Q(x)-z)^{a-3/2} }{Q(x)^{2a-1}} \,\mathrm{d}z
\\
&&\quad= 1- \biggl( \frac{Q(x)}{Q(y)} \biggr)^{2a-1}.\nonumber
\end{eqnarray}
We now observe that for $p>0$,
%
\begin{equation}
\label{eq-MVT1}
v^p-u^p \leq \max\{p,1\} v^{p-1}|v-u|
\qquad\mbox{for $v\geq u\geq0$}
\end{equation}
(for $0<p\leq1$, this is simple algebra and for $p>1$, this follows from
applying the mean value theorem to the function $t\mapsto t^p$). Next, since
$|Q'(x)| = |4-8x| \leq4$, the mean value theorem implies that
%
\begin{equation}
\label{eq-MVT2}
|Q(y)-Q(x)| \leq 4 |y-x| \qquad\mbox{for $x,y\in[0,1]$}.
\end{equation}
Finally,
for $0<Q(x)\leq Q(y)$, equations (\ref{eq.Qratio})--(\ref{eq-MVT2}) imply that
\begin{eqnarray*}
\frac{1}{2}\int_0^1|p(x,z)-p(y,z)| \,\mathrm{d}z & \leq&
\frac{ Q(y)^{2a-1}-Q(x)^{2a-1}}{Q(y)^{2a-1}} \\
& \leq& \frac{\max\{(2a-1),1\} |Q(y)-Q(x)|}{Q(y)}
\\
& \leq& \frac{[(2a-1)+1] 4 |y-x|}{Q(y)} .
\end{eqnarray*}
This proves the lemma.
\end{pf}

We can now apply Theorem \ref{prop-WTVgen}(b) as follows.
Let $\mu=\delta_x$
(point mass at $x$) and let $\pi_a$ be the equilibrium $\beta_{a,a}$
distribution. Also, let $\lambda$ be Lebesgue measure and let the function
$h(\cdot)$ be $Q(\cdot)/(16a)$. Lemma \ref{lem-logwass} then proves
condition (i) of Theorem \ref{prop-WTVgen}(b).
For condition (ii), we need to estimate
$\pi_a(\{y\in[0,1] : h(y)\leq\epsilon\})$ for small positive $\epsilon$.
Let $A=16a$. Observe that if $Q(y)/A\leq\epsilon$
and $0\leq y\leq1/2$, then $A\epsilon\geq4y(1-y)\geq4y(1/2)$, so
$y\leq A\epsilon/2$. Similarly, if $Q(y)/A\leq\epsilon$ and
$1/2\leq y\leq1$, then $y\geq1-A\epsilon/2$. Therefore, for $a\geq1$
and $0<\epsilon\leq1/A$, we have
%
\begin{eqnarray}
\nonumber
&&\pi_a\bigl(\{y\in[0,1] \dvtx h(y)\leq\epsilon\}\bigr) \\
\nonumber
&&\quad =
\pi_a([0,A\epsilon/2]) + \pi_a([1-A\epsilon/2,1]) \\
\nonumber
&&\quad =  2 \pi_a([0,A\epsilon/2])
\qquad\mbox{(since $\pi_{a}$ is symmetric about $1/2$)}
\\
\label{eq-logint1}
&&\quad = \tilde{K}_a \int_0^{A\epsilon/2}t^{a-1}(1-t)^{a-1} \,\mathrm{d}t
\qquad\bigl(\mbox{where $\tilde{K}_a=\Gamma(2a)/\Gamma(a)^2$}\bigr) \\
\label{eq-logint2}
&&\quad  \leq \tilde{K}_a \int_0^{A\epsilon/2}t^{a-1} \,\mathrm{d}t \\
\nonumber
&&\quad  =  \frac{\tilde{K_a} (8a\epsilon)^a}{a} .
\label{eqLIB}
\end{eqnarray}
Therefore, equation (\ref{eq.QTVprop2}) holds with $q=a$,
$B=\tilde{K_a}8^a a^{a-1}$ and $\epsilon_0 = 1/(16a)$.
For $1/2<a<1$, everything is the same except that we use the bound
$(1-t)^{a-1}\leq2^{1-a}$ for $0<t\leq1/2$
in the integrand of (\ref{eq-logint1}), obtaining an extra multiplicative
factor of $2^{1-a}$ in equation (\ref{eq-logint2}) and hence
$B = 2\tilde{K_a}4^a a^{a-1}$. We have thus shown that
Theorem \ref{thm-logconv} follows from Theorem \ref{prop-WTVgen}(b).

\section*{Acknowledgements}
We are grateful to David Steinsaltz for very helpful discussions,
to Jeffrey Rosenthal for pointers to the literature and to the referees
for useful suggestions which helped us improve this paper.
The research of N. Madras was supported in part by a Discovery Grant
from NSERC of Canada.

\printhistory

\end{document}